 \documentclass[11pt]{amsart}
\usepackage[active]{srcltx}

\usepackage{pb-diagram}

\usepackage[usenames,dvipsnames,svgnames,table]{xcolor}

\usepackage{mathrsfs}
\usepackage{amsmath}
\usepackage{amssymb}
\usepackage{amscd}
\usepackage{amsthm}
\usepackage[latin1]{inputenc}
\usepackage{graphics}
\usepackage{varioref}

\labelformat{enumi}{(#1)}

\newtheorem{teo}{Theorem} [section]
\newtheorem{defin}{Definition} [section]
\newtheorem{remark}{Remark}[section]
\newtheorem{prop}{Proposition}[section]
\newtheorem{cor}{Corollary}[section]
\newtheorem{lemma}{Lemma}[section]

\newcounter{yuppo}
\newtheorem{yuppi}{Theorem}[yuppo]

\newcommand{\Keler} {K\"{a}hler }

\newcommand{\End}{\operatorname{End}}

\newcommand{\cds}{\cdots}
\newcommand{\cd}{\cdot}
\renewcommand{\setminus}{-}
\newcommand{\om}{\omega}

\renewcommand{\phi}{\varphi}
\newcommand{\cinf}{C^\infty}

\newcommand{\ra}{\rightarrow}

\newcommand{\C}{\mathbb{C}}
\newcommand{\R}{\mathbb{R}}

\newcommand{\SU} {\operatorname{SU}}

\newcommand{\Sl}{\operatorname{SL}}

\newcommand{\Gl}{\operatorname{Gl}}

\newcommand{\restr}[1]          {\vert_{#1}}

\newcommand{\Ad}{\operatorname{Ad}}
\newcommand{\ad}{{\operatorname{ad}}}

\newcommand{\ga}{\gamma}

\newcommand{\meno}{^{-1}}

\newcommand{\PP}{\mathbb{P}}

\newcommand{\enf}{\emph}

\newcommand{\desudt}[1] []      {\dfrac {\mathrm {d} #1 }{\mathrm {dt}}}
\newcommand{\desudtzero}        {\desudt \bigg \vert _{t=0} }

\renewcommand{\root}{\Delta}
\newcommand{\conv} {\operatorname{conv}}
\newcommand{\liu}{\mathfrak{u}}

\newcommand{\lia}{\mathfrak{a}}

\newcommand{\liek}{\mathfrak{k}}
\newcommand{\lier}{\mathfrak{r}}

\newcommand{\lieg}{\mathfrak{g}}

\newcommand{\liep}{\mathfrak{p}}
\newcommand{\lieq}{\mathfrak{q}}
\newcommand{\liez}{\mathfrak{z}}
\newcommand{\lies}{\mathfrak{s}}

\newcommand{\liem}{\mathfrak{m}}
\newcommand{\lien}{\mathfrak{n}}

\newcommand{\lipo} {\liep_1}

\newcommand{\liko} {\liek_1}

\newcommand{\lipt} {\liep_0}
\newcommand{\likt} {\liek_0}

\newcommand{\im}{\operatorname{Im}}

\newcommand{\la}{\lambda}

\newcommand{\alfa}{\alpha}
 \newcommand{\vacuo}{\emptyset}
\newcommand{\OO}{\mathcal{O}} 
\renewcommand{\c}{{\widehat{\OO}}} 
\newcommand{\polp}{{P}} 
\newcommand{\ml}{\operatorname{Max}} 
\newcommand{\ext}{\operatorname{ext}} 
\newcommand{\sx}{\langle} 
\newcommand{\xs}{\rangle}
\newcommand{\relint}{\operatorname{relint}} 
\newcommand{\Weyl}{W} 
\newcommand{\cchamber}{\overline{C}}
\newcommand{\chamber}{{C}}
\newcommand{\Crit}{\operatorname{Crit}} 
\newcommand{\roots}{\Delta} 
\newcommand{\simple}{\Pi} 
\newcommand{\faces}{\mathscr{F}(\c)} 
\newcommand{\facesp}{\mathscr{F}(\polp)} 

\newcommand{\spam}{\operatorname{span}}

\newcommand{\HF}{{H_F}}

\newcommand{\CF}{C_F}

\newcommand{\liea}{\mathfrak{a}}

\newcommand{\CFI}{C_F^{\HF}}

\newcommand{\noparty}[1]{}
\newcommand{\changed}[1]{{#1}}

\newcommand{\scalo}{\sx \, , \, \xs}

\newcommand{\metrica}{(\, , \, )}

\newcommand{\mup}{\mu_\liep}
\newcommand{\mupb}{\mu_\liep^\beta}

\begin{document}

\author{Leonardo Biliotti, Alessandro Ghigi and Peter Heinzner}

\title{Polar orbitopes}

   \address{Universit\`{a} di Parma} \email{leonardo.biliotti@unipr.it}
   \address{Universit\`a di Milano Bicocca}
   \email{alessandro.ghigi@unimib.it}

   \address{Ruhr Universit\"at Bochum} \email{peter.heinzner@rub.de}

  \thanks{The first author was partially supported by GNSAGA of INdAM.
    The second author was partially supported by GNSAGA of INdAM and
    by PRIN 2009 MIUR ''Moduli, strutture geometriche e loro
    applicazioni''. The third author was partially supported by
    DFG-priority program SPP 1388 (Darstellungstheorie)}


\maketitle

\begin{abstract}
  We study \emph{polar orbitopes}, i.e. convex hulls of orbits of a
  polar representation of a compact Lie group. They are given by
  representations of $K$ on $\liep$, where $K$ is a maximal compact
  subgroup of a real semisimple Lie group $G$ with Lie algebra $\lieg
  = \liek \oplus \liep$. The face structure is studied by means of the
  gradient momentum map and it is shown that every face is exposed and
  is again a polar orbitope. Up to conjugation the faces are
  completely determined by the momentum polytope. There is a tight
  relation with parabolic subgroups: the set of extreme points of a
  face is the closed orbit of a parabolic subgroup of $G$ and for any
  parabolic subgroup the closed orbit is of this form.
\end{abstract}

\tableofcontents{}

\section{Introduction}

If $K$ is a compact group and $K\ra \Gl(V)$ is a real representation,
the convex hull of a $K$-orbit is called an \emph{orbitope}
\cite{sanyal-sottile-sturmfels-orbitopes}.  If $V$ is provided with a
$K$-invariant scalar product, the representation is said to be
\emph{polar} if there is a linear subspace $S \subset V$ that
intersects perpendicularly all the orbits of $K$.  An important class
of examples is given by the adjoint representations of compact Lie
groups.  In \cite{biliotti-ghigi-heinzner-1-preprint} we studied the
orbitopes of these actions.  They are equivariantly isomorphic to
Satake-Furstenberg compactifications of symmetric spaces of type
$K^\C/K$. One homeomorphism has been described in algebraic terms in
\cite{koranyi-remarks}. Another homeomorphism has been constructed in
\cite{biliotti-ghigi-2} (in the case of an integral orbit) using
integration of the momentum map on a flag manifold.  This geometric
construction was developed by Bourguignon, Li and Yau in the case of
$\PP^n$.

In the present paper we study the orbitopes of a polar representation
of a compact group.  Let $G$ be a real connected semisimple Lie group
and let $\lieg = \liek \oplus \liep$ be a Cartan decomposition of its
Lie algebra.  Let $K$ be a the maximal compact subgroup with Lie
algebra $\liek$. Then the adjoint action of $K$ preserves $\liep$ and
its restriction to $\liep$ is a polar representation.  By a theorem of
Dadok \cite[Prop. 6]{dadok-polar} if $V$ is any polar representation
of a group $K_1$, there is a semisimple Lie group $G$ such that $V$
can be identified with $ \liep$ so that the orbits of $K_1$ coincide
with the orbits of $\Ad K$ on $\liep$.  Therefore to understand the
orbitopes of polar representations it is sufficient to study the
$K$-orbitopes on $\liep$.

The study of these orbitopes is also needed in order to generalize the results in
\cite{biliotti-ghigi-2} to general symmetric spaces and this is one of
the motivations for our work.

Our set up is the following.
Let $U$ be compact Lie group and let $U^\C$ be its complexification.
A closed subgroup $G\subset U^\C$ is called \enf{compatible} if
$G=K\cd \exp\liep$ where $K:=G\cap U$ and $\liep:= \lieg \cap i\liu$.
It follows that $K$ is a maximal compact subgroup of $G$ and that
$\lieg = \liek \oplus \liep$.  $K$ acts on $\lieg$ by the adjoin
action and $\liep$ is invariant.  Therefore we get an action of $K$ on
$\liep$.  The objects that we wish to study are the orbits of this
action and their convex hulls. 
It is easy to see that one can reduce to the case in which $U$ and $G$
are semisimple (see \S \ref{faces-2}).  If $\OO \subset \liep$ is a
$K$-orbit, we denote by $\c$ its convex hull.  We will assume
throughout the paper that $G$ is connected.  It is a fundamental fact
that the action of $K$ on $\OO$ extends to an action of $G$, see e.g.
\cite[Prop. 6]{heinzner-stoetzel}.  If $\lia\subset \liep$ is a
maximal subalgebra, then by Kostant convexity theorem
\cite{kostant-convexity}, the orthogonal projection of $\OO$ onto
$\lia$ is a convex polytope $P$ given by the convex hull of a Weyl
group orbit.  In particular the Weyl group acts on the set $\facesp$
of faces of $P$ and similarly $K$ acts on the set $\faces$ of faces of
$\c$.

Our main result is the following.
\begin{teo}
  \label{main} Let $P \subset \lia$ be the momentum polytope
  associated to $\OO$. If $\sigma $ is a face of $P$ and
  $K^{\sigma^\perp}$ is the centralizer of the normal space
  $\sigma^\perp\subset \lia$, then $K^{\sigma^\perp}\cd \sigma$ is a
  face of $\c$. Moreover the map $\sigma \mapsto K^{\sigma^\perp}\cd
  \sigma$ induces a bijection between $\facesp/W$ and $\faces/K$.
\end{teo}
The correspondence between $\faces/K$ and $\facesp/W$ holds for a
general polar representation, see Remark \ref{polare} at
p. \pageref{polare}.  Applied to the case $G=U^\C$ this
\changed{theorem} gives the results proven in
\cite{biliotti-ghigi-heinzner-1-preprint}.  \changed{The setting of
  the present paper is more general than the one considered
  there. The pairs $(G,K)$ with $G$ compatible contain all Riemannian
  symmetric pairs of noncompact type, while the pairs $(U^\C, U)$
  correspond to symmetric pairs of type IV \cite[p. 516]{helgason}.}
The particular cases $U=\SU(n)$, $G=\Sl(n, \R)$ and $
U=\operatorname{SO}(n)$, $G=\operatorname{SO}(n, \C)$ have been
considered in \cite{sanyal-sottile-sturmfels-orbitopes}.
\changed{The case where $\OO$ can be realized as the Shilov boundary of a Hermitian symmetric domain has been studied in \cite[Prop. 2.1]{clerc-neeb}.}
%


We outline the main steps of the proof.

Among the faces of a convex set are the exposed faces (see \S
\ref{pre-convex}). In the case of $\c$ the study of these faces is
equivalent to the understanding of the height functions on $\OO$ (\S
\ref{faces1}). This is a classical subject, going back to the paper
\cite{duistermaat-kolk-varadarajan} by Duistermaat, Kolk and
Varadarajan and to Heckman's thesis \cite{ heckman-thesis}.  The
results are very efficiently described in the language of the
\emph{gradient momentum map} (which is recalled in \S
\ref{subsection-gradient-moment}).  The set of extreme points $\ext F$
of an exposed face $F$ is connected and is an orbit of a centralizer
$K^\beta \subset K$, where $\beta $ is an element of $\liep$
(Proposition \ref {massimo-connesso}).  In general the group $K^\beta$
is not connected.  An inductive argument shows that any face $F\subset
\c$ (not necessarily exposed) is an orbitope of the centralizer
$K^\lies$ of some subalgebra $\lies \subset \liep$ (Proposition
\ref{facciona-orbita}). If $\lia \subset \liep$ is a maximal
subalgebra containing $\lies$, we show that $F\cap \lia$ is a face of
the momentum polytope and that $F\cap \lia$ determines $F$ (Proposition
\ref{proiezione-intersezione}). Here we use in an essential way the
Kostant convexity theorem.

An important conclusion is that all faces of $\c$ are exposed (Theorem
\ref{tutte-esposte}). This answers Question 1 of
\cite{sanyal-sottile-sturmfels-orbitopes} for polar orbitopes.
Next \changed{recall that the $K$-action on $\OO$ extends to an action of the group $G$ (see \S \ref{coadjoint-orbits} below).}  We analyze the influence of the $G$-action on the geometry of the
extreme points of the faces (\S \ref{parabolic-section}). It turns out
that there is a strong link between the parabolic subgroups of $G$ and
the faces of $\c$. In \ref{parabolic-section} we show the following.
\begin{teo}\label{main-2}
  The set $\{\ext F: F $ a nonempty face of $\c \}$ coincides with the
  set of all closed orbits of parabolic subgroups of $G$.
\end{teo}
Using these results we finally set up the correspondence between the
faces of $\c$ and the faces of $P$ and prove Theorem \ref{main} (\S
\ref {polysection}).


In the final section we briefly explain how the
boundary of $\c$ is stratified by face type and how the Satake
combinatorics can be used to describe the faces of the orbitope in
terms of root data.

{\bfseries \noindent{Acknowledgements.}}  The first two authors are
grateful to the Fakultät für Mathematik of Ruhr-Universität Bochum for
the wonderful hospitality. \changed{We also would like to thank the
  referees for helpful comments.}

 \section{Preliminaries}

 \subsection{Convex geometry}

 \label{pre-convex}

 It is useful to recall a few definitions and results regarding convex
 sets (see e.g. \cite{schneider-convex-bodies} and \cite[\S
 1]{biliotti-ghigi-heinzner-1-preprint}).  Let $V$ be a real vector
 space with a scalar product $\scalo$ and let $E\subset V$ be a
 \changed{compact} convex subset.  The \emph{relative interior} of
 $E$, denoted $\relint E$, is the interior of $E$ in its affine hull.
 A face $F$ of $E$ is a convex subset $F\subset E$ with the following
 property: if $x,y\in E$ and $\relint[x,y]\cap F\neq \vacuo$, then
 $[x,y]\subset F$.  The \emph{extreme points} of $E$ are the points
 $x\in E$ such that $\{x\}$ is a face. Since  $E$ is compact the faces
 are closed \cite[p. 62]{schneider-convex-bodies}.  A face distinct
 from $E$ and $\vacuo$ will be called a \enf{proper face}.  The
 \enf{support function} of $E$ is the function $ h_E : V \ra \R$, $
 h_E(u) = \max_{x \in E} \sx x, u \xs$.  If $ u \neq 0$, the
 hyperplane $H(E, u) : = \{ x\in E : \sx x, u \xs = h_E(u)\}$ is
 called the \enf{supporting hyperplane} of $E$ for $u$. The set
   \begin{gather}
     \label{def-exposed}
     F_u (E) : = E \cap H(E,u)
   \end{gather}
   is a face and it is called the \enf{exposed face} of $E$ defined by
   $u$. 
 In general not all faces of a convex subsets are exposed.
 A simple example is given by the convex hull of a closed disc and a
 point outside the disc: the resulting convex set is the union of the
 disc and a triangle. The two vertices of the triangle that lie on the
 boundary of the disc are non-exposed 0-faces.

 \begin{lemma}
[\protect{\cite[Lemma 3]{biliotti-ghigi-heinzner-1-preprint}}]
\label{ext-facce}
If $F$ is a face of a convex set $E$, then $\ext F = F \cap \ext E$.
 \end{lemma}
 \begin{lemma}
   \label{convex-orbit}
   If $G$ is a compact group and $V$ is a representation space of $G$ define
   \begin{gather*}
     \rho : V \ra V^G \qquad \rho(v) := \int_G gx \, dg
   \end{gather*}
   where $dg$ denotes the Haar measure on $G$. Then $V=V^G\oplus \ker
   \rho$. If $x\in V$ and $x=x_0 + x_1$ in this decomposition, then
   \begin{enumerate}
   \item   $G\cd x = x_0 + G\cd x_1$;
 \item  $\conv (G\cd x) = x _0 + \conv (G\cd x_1)$;
 \item  $x_0$ is the unique fixed point of $G$ contained in $\conv (G\cd
   x)$;
 \item $x_0 \in \relint \conv ( G\cd x)$.
   \end{enumerate}
 \end{lemma}
 \begin{proof}
   That $V=V^G\oplus \ker \rho$ follows from the fact that $\im \rho =
   V^G$ and $\rho^2 = \rho$. (a) and (b) are immediate.  Since
   $x_0=\rho(x) $, it follows from the definition of $\rho$ that
   $x_0\in \conv (G\cd x)$.  If $y \in \conv (G\cd x)$ is another fixed
   point, then $y_0=x_0 $ and $y_1 \in \ker \rho \cap V^G$. Hence
   $y_1=0$ and $y=x_0$. This proves (c).  By Theorem
   \ref{schneider-facce} there is a unique face $F \subset \conv (G\cd
   x) $ such that $x_0\in \relint F$. Since $\conv (G\cd x)$ is
   $G$-invariant and $ x_0 $ is fixed by $G$, also $F$ is
   $G$-invariant, and hence also $\ext F$. Since $\ext F \subset \ext
   (\conv (G\cd x) ) = G\cd x$, it follows that $\ext F = G\cd x$ and
   hence that $F = \conv (G\cd x)$.
 \end{proof}

 \begin{lemma}
[\protect{\cite[Prop. 5]{biliotti-ghigi-heinzner-1-preprint}}]
   \label{u-cono}
   If $F \subset E$ is an exposed face, the set $\CF : = \{ u\in V:
   F=F_u(E) \}$ is a convex cone. If $G$ is a compact subgroup of
   $O(V)$ that preserves both $E$ and $F$, then $\CF$ contains a fixed
   point of $G$.
 \end{lemma}
 \begin{teo} [\protect{\cite[p. 62]{schneider-convex-bodies}}]
   \label{schneider-facce} If $E$ is a compact convex set and
   $F_1,F_2$ are distinct faces of $E$ then $\relint F_1 \cap \relint
   F_2=\vacuo$. If $G$ is a nonempty convex subset of $ E$ which is
   open in its affine hull, then $G \subset\relint F$ for some face
   $F$ of $E$. Therefore $E$ is the disjoint union of \changed{the
     relative interiors of its} faces.
 \end{teo}

 \begin{lemma}
 [\protect{\cite[Lemma 7]{biliotti-ghigi-heinzner-1-preprint}}]
   If $E$ is a compact convex set and $F\subsetneq E$ is a face, then
   $\dim F < \dim E$.
 \end{lemma}

 \begin{lemma}
 [\protect{\cite[Lemma 8]{biliotti-ghigi-heinzner-1-preprint}}]
   \label{face-chain}
   If $E$ is a compact convex set and $F\subset E$ is a face, then
   there is a chain of faces $ F_0=F \subsetneq F_1 \subsetneq \cds
   \subsetneq F_k=E $ which is maximal, in the sense that for any $i$
   there is no face of $E$ strictly contained between $F_{i-1}$ and
   $F_i$.
 \end{lemma}
 \begin{lemma}
 [\protect{\cite[Lemma 9]{biliotti-ghigi-heinzner-1-preprint}}]
 \label{micro-convesso} If $E$ is a convex subset of
   $\R^n$, $M\subset \R^n$ is an affine subspace and $F\subset E$ is a
   face, then $F\cap M $ is a face of $E\cap M$.
 \end{lemma}
%

 \subsection{Compatible subgroups}
\label{comp-subgrous}
(See \cite{heinzner-schwarz-stoetzel, heinzner-stoetzel-global}.)  If
 $G$ is a Lie group with Lie algebra $\lieg$ and $E, F \subset \lieg$,
 we set
 \begin{gather*}
   E^F :=     \{\eta\in E: [\eta, \xi ] =0, \forall \xi \in F\} \\
   G^F = \{g\in G: \Ad g (\xi ) = \xi, \forall \xi \in F\}.
 \end{gather*}
 If $F=\{\beta\}$ we write simply $E^\beta$ and $G^\beta$. Let $U$ be
 compact Lie group. Let $U^\C$ be its universal complexification which
 is a linear reductive complex algebraic group. We
   denote by $\theta$ both the conjugation map $\theta : \liu^\C \ra
   \liu^\C$ and the corresponding group isomorphism $\theta : U^\C \ra
   U^\C$.  Let $f: U \times i\liu \ra U^\C$ be the diffeomorphism
 $f(g, \xi) = g \exp \xi$.  Let $G\subset U^\C$ be a closed
 subgroup. Set $K:=G\cap U$ and $\liep:= \lieg \cap i\liu$.  We say
 that $G$ is \enf{compatible} if $f (K \times \liep) = G$. The
 restriction of $f$ to $K\times \liep$ is then a diffeomorphism onto
 $G$. It follows that $K$ is a maximal compact subgroup of $G$ and
 that $\lieg = \liek \oplus \liep$.  Note that $G$ has finitely many
 connected components. Since $U$ can be embedded in $\Gl(N,\C)$ for
 some $N$, and any such embedding induces a closed embedding of
 $U^\C$, any compatible subgroup is a closed linear group. Moreover
 $\lieg$ is a real reductive Lie algebra, hence $\lieg =
 \liez(\lieg)\oplus [\lieg, \lieg]$. Denote by $G_{ss}$ the analytic
 subgroup tangent to $[\lieg, \lieg]$. Then $G_{ss}$ is closed and
 $G=Z(G)^0\cd G_{ss}$ \cite[p. 442]{knapp-beyond}.
 \begin{lemma}\label{lemcomp}
   \begin{enumerate}
   \item \label {lemcomp1} If $G\subset U^\C$ is a compatible
     subgroup, and $H\subset G$ is closed and $\theta$-invariant,
     then $H$ is compatible if and only if $H$ has only finitely many connected components.
   \item \label {lemcomp2} If $G\subset U^\C$ is a connected
     compatible subgroup, then $G_{ss}$ is compatible.
     \item \label{lemcomp3} If $G\subset U^\C$ is a compatible
       subgroup, and $E\subset \liep$ is any subset, then $G^E$ is
       compatible.
   \end{enumerate}
 \end{lemma}
 \begin{proof}
   \ref{lemcomp1} \changed{This follows from the more general
     observation that a closed $\theta$-invariant subgroup $G\subset
     U^\C$ is compatible if and only if it has finitely many connected
     components. This is proven in Lemma 1.1.3 in
     \cite[p.14]{miebach}. For the reader's convenience we recall the
     argument. If $G$ is compatible, then it retracts onto $K$, which
     is compact and therefore has finitely many connected
     components. Conversely assume that $G/G^0$ be finite.  Since $G$
     is closed, $f(K\times \liep)$ is a closed subset of $G$.  Since
     $G$ is $\theta$-invariant, $f(K\times \liep) $ has the same
     dimension as $G$ and is therefore also open. Therefore it
     contains $G^0$ and is a union of connected components of $G$.
     Given $g\in G$ write $g=u\exp\xi$ with $u\in U$ and $\xi \in i
     \liu$. Then $g\theta(g\meno) = \exp(2 \Ad(u)\xi) $ and since
     $G/G^0$ is finite there is a natural number $N>0$ such that
     $\bigl ( g \theta(g\meno) \bigr)^N = \exp(2 N \Ad(u)\xi) \in
     G^0$. Hence $\Ad(u)\xi \in \liep$, $u = \exp (-\Ad(u)\xi) g \in
     G\cap U = K$ and $\xi \in \liep$. } \ref{lemcomp2}. Since
   $[\lieg, \lieg]$ is $\theta$-invariant and $G_{ss}$ is connected,
   $G_{ss}$ is $\theta$-invariant. Since it is also closed, it is
   compatible by \ref{lemcomp1}.  \ref{lemcomp3} see \cite[Proposition
   7.25 p. 452]{knapp-beyond}.
 \end{proof}
 Let $\scalo$ be a fixed $U$-invariant scalar product on $\liu$. We
 use it to identifiy $\liu \cong \liu^*$.  We also denote by $\scalo$
 the scalar product on $i\liu$ such that multiplication by $i$ be an
 isometry of $\liu$ onto $i\liu$. One can define an
   $\R$-bilinear form $B$ on $\liu^\C$ by imposing $B(\liu, i
   \liu)=0$, $B= -\scalo$ on $\liu$ and $B= \scalo$ on $i\liu$.  Then
   $B$ is $\Ad U^\C$-invariant and nondegenerate. \label{def-B}

\subsection{Parabolic subgroups}
(See e.g. \cite[p. 28ff]{borel-ji-libro}, \cite{knapp-beyond}.)  If $G
\subset U^\C$ is compatible, $\lieg = \liek \oplus \liep$ is
reductive.  A subalgebra $\lieq \subset \lieg$ is \enf{parabolic} if
$\lieq^\C$ is a parabolic subalgebra of $\lieg^\C$.  One way to
describe the parabolic subalgebras of $\lieg$ is by means of
restricted roots.  If $\lia \subset \liep$ is a maximal subalgebra,
let $\roots(\lieg, \lia)$ be the (restricted) roots of $\lieg$ with
respect to $\lia$, let $\lieg_\la$ denote the root space corresponding
to $\la$ and let $\lieg_0 = \liem \oplus \lia$, where $\liem =
\liez_\liek(\lia)$.  Let $\simple \subset \roots(\lieg, \lia)$ be a
base and let $\roots_+$ be the set of positive roots. If $I\subset
\simple$ set $\roots_I : = \spam(I) \cap \roots$. Then
\begin{gather}
  \label{para-dec}
  \lieq_I:= \lieg_0 \oplus \bigoplus_{\la \in \roots_I \cup \roots_+}
  \lieg_\la
\end{gather}
is a parabolic subalgebra. Conversely, if $\lieq \subset \lieg$ is a
parabolic subalgebra, then there are a maximal subalgebra $\lia
\subset \liep$ contained in $\lieq$, a base $\simple \subset
\roots(\lieg, \lia)$ and a subset $I\subset \simple $ such that $\lieq
= \lieq_I$.  We can further introduce
\begin{gather}
\label{notaz-I}
\begin{gathered}
  \lia_I : = \bigcap_{\la \in I} \ker \la \qquad \lia^I := \lia_I^\perp \\
  \lien_I = \bigoplus_{\la \in \roots_+ \setminus \roots_I} \lieg_\la
  \qquad \liem_I : = \liem \oplus \lia^I \oplus \bigoplus_{\la \in
    \roots_I}\lieg_\la.
\end{gathered}
\end{gather}
Then $\lieq_I = \liem_I \oplus \lia_I \oplus \lien_I$. Since
$\theta\lieg _ \la = \lieg_{-\la}$, it follows that $ \lieq_I \cap
\theta\lieq_I = \lia_I \oplus \liem_I$.  This latter \changed{Lie algebra} coincides
with the centralizer of $\lia_I$ in $\lieg$. It is a \enf{Levi factor}
of $\lieq_I$ and
\begin{gather}
  \label{liaI}
  \lia_I =\liez (\lieq_I \cap \theta\lieq_I) \cap \liep.
\end{gather}
Another way to describe parabolic subalgebras of $\lieg$ is the
following.  If $\beta \in \liep$, the endomorphism $\ad \beta \in \End
\lieg$ is diagonalizable over $\R$. Denote by $ V_\la (\ad \beta) $ the
eigenspace of $\ad \beta$ corresponding to the eigenvalue $\la$.  Set
\begin{gather*}
  \lieg^{\beta+}: = \bigoplus_{\la \geq 0} V_\la (\ad \beta).
\end{gather*}
\begin{lemma}
  \label{para-beta}
  For any $\beta$ in $ \liep$, $\lieg^{\beta+}$ is a parabolic
  subalgebra of $\lieg$.  If $\lieq \subset \lieg$ is a parabolic
  subalgebra, there is some vector $\beta \in \liep$ such that $\lieq
  = \lieg^{\beta+}$.  The set of all such vectors is an open convex
  cone in
$\liez(\lieq \cap \theta\lieq) \cap \liep$.
\end{lemma}
\begin{proof}
  Given $\beta$ choose a maximal subalgebra $\lia$ containing $\beta$
  and a base $\simple \subset \roots(\lieg, \lia)$ such that $\beta$
  lies in the closure of the positive Weyl chamber. Then
  $\lieg^{\beta+} = \lieq _I$ with $I:=\{\la \in \simple:
  \la(\beta)=0\}$. This proves the first assertion.  To prove the
  second fix a parabolic subalgebra $\lieq$ and set $\Omega:=\{\beta
  \in \liep: \lieg^{\beta+} = \lieq\}$. 
  Let $\lia $ be any maximal subalgebra of $ \liep$ contained in
  $\lieq$.
Then $\lieq = \lieq_I$ for some $I \subset \simple$ and
  \begin{gather}
    \label{Omegalia}
    \Omega \cap \lia = \{\beta \in \lia_I : \la (\beta) > 0 \text{
      for } \la \in \simple\setminus I\}.
  \end{gather}
  Thus $\Omega \cap \lia$ is a nonempty open convex cone in $\lia_I$.
  Therefore $\Omega \neq \vacuo$, which proves the second assertion.
  By \eqref{liaI} $\lia_I = \liez(\lieq \cap \theta\lieq) \cap \liep$,
  so $\Omega\cap \lia$ is an open convex cone in $\liez(\lieq \cap
  \theta\lieq) \cap \liep$.  \changed{Moreover for any $\beta \in
    \Omega$, $\lia \subset \lieq \cap \theta(\lieq) =
    \lieg^\beta$. Thus $[\beta, \lia ] = 0$, hence $\beta \in \lia$.
    So $\Omega \subset \lia$, i.e. $\Omega = \Omega \cap \lia$.
  }
%
\end{proof}
A \enf {parabolic subgroup} of $G$ is a subgroup of the form $Q =
N_G(\lieq)$ where $\lieq $ is a parabolic subalgebra of $\lieg$.
Equivalently, a parabolic subgroup of $G$ is a subgroup of the form
$P\cap G$ where $P$ is parabolic subgroup of $G^\C$ and $\liep$ is the
complexification of a subspace $\lieq \subset \lieg$.  If $\beta
\in \liep$ set
\begin{gather*}
\begin{gathered}
  G^{\beta+} :=\{g \in G : \lim_{t\to - \infty} \exp({t\beta}) g
  \exp({-t\beta}) \text { exists} \}\\
  R^{\beta+} :=\{g \in G : \lim_{t\to - \infty} \exp({t\beta})
  g \exp({-t\beta}) =e \}
\end{gathered}
\qquad
  \lier^{\beta+}: = \bigoplus_{\la > 0} V_\la (\ad \beta).
\end{gather*}
Note that $ \lieg^{\beta+}= \lieg^\beta \oplus \lier^{\beta+}$.
\begin{lemma}
  $G^{\beta +} $ is a parabolic subgroup of $G$ with Lie algebra
  $\lieg^{\beta+}$.  Every parabolic subgroup of $G$ equals
  $G^{\beta+}$ for some $\beta \in \liep$.  $R^{\beta+}$ is the
  unipotent radical of $G^{\beta+}$ and $G^\beta$ is a Levi factor.
\end{lemma}
\begin{proof}
  It is easy to check that $G^{\beta+}$ is a subgroup and that
  $G^{\beta + } = (G^\C)^{\beta+} \cap G$. Therefore it is enough to
  prove that $(G^\C) ^{\beta+}$ is parabolic. In other words we can
  assume that $G$ is a complex reductive group.  If $X \in \lieg$,
  then
  \begin{gather*}
    \exp(t\beta) \exp X \exp (-t\beta) = \exp (\Ad(\exp (t\beta)) \cd
    X ) = \exp( e^{t\ad \beta} \cd X)
  \end{gather*}
  where $e^{t\ad \beta}$ denotes the exponential in $\End (\lieg)$.
  Let $\Omega \subset \lieg$ be a neighbourhood of $0$ such that
  $\exp$ is \changed{a diffeomorphism} on $\Omega$. If $X \in \Omega$, then $\exp X
  \in R^{\beta+}$ if and only if $\lim_{t\to - \infty} e^{t\ad \beta} \cd X =0$
  if and only if $X \in \lier^{\beta+}$. This shows that $R^{\beta +}$ is locally
  closed, hence closed \cite[Prop. 2.11 p. 119]{helgason}. Next
  observe that if $g \in G^{\beta+}$, and
  \begin{gather*}
    a : = \lim_{t\to -\infty} \exp(t\beta) g \exp(-t\beta)
  \end{gather*}
  then $a \in G^\beta \subset G^{\beta+}$ and $a\meno g\in
  R^{\beta+}$. Therefore $G^{\beta+}$ is the product of the two closed
  subgroups $G^\beta $ and $R^{\beta+}$ and $G^\beta \cap R^{\beta+} =
  \{e\}$. It follows that $G^{\beta+}$ is a Lie subgroup of $G$
  tangent to $\lieg^{\beta+}$. \changed{Since we are now assuming that
    $G$ is complex, then} it is well-known that $G^{\beta+}$ is
closed and parabolic since its Lie algebra is parabolic.
\end{proof}

\subsection{Gradient momentum map}
\label{subsection-gradient-moment}
Let $(Z, \om)$ be a \Keler manifold. Assume that $U^\C$ acts
holomorphically on $Z$, that $U$ preserves $\om$ and that there is a
momentum map $\mu: Z \ra \liu$.  If $\xi \in \liu$ we denote by $\xi_Z$
the induced vector field on $Z$ and we let $\mu^\xi \in \cinf(Z)$ be
the function $\mu^\xi(z) := \sx \mu(z),\xi\xs$.  That $\mu$ is the
momentum map means that it is $U$-equivariant and that $d\mu^\xi =
i_{\xi_Z} \om$.

Let $G \subset U^\C$ be compatible.
If $z \in Z$, let $\mup (z) \in \liep$ denote $-i$ times the component
of $\mu(z)$ in the direction of $i\liep$.  In other words we require
that $\sx \mup (z) , \beta \xs = - \sx \mu(z) , i\beta\xs$ for any
$\beta \in \liep$.  (Recall that multiplication by $i$ is an isometry
of $\liu $ onto $i \liu$.) We have thus defined the \emph{gradient
  momentum map}
\begin{gather*}
  \mu_\liep : Z \ra \liep.
\end{gather*}
Let $\mup^\beta \in \cinf(Z)$ be the function $ \mup^\beta(z) = \sx
\mup(z) , \beta\xs = \mu^{-i\beta}(z)$.  Let $\metrica$ be the \Keler
metric associated to $\om$, i.e. $(v, w) = \om (v, Jw)$. Then
$\beta_Z$ is the gradient of $\mup^\beta$. If $X \subset Z$ is a
locally closed $G$-invariant submanifold, then $\beta_X$ is the
gradient of $\mup^\beta \restr{X}$ with respect to the induced
Riemannian structure on $X$.

\begin{teo}
  [Slice Theorem \protect{\cite[Thm. 3.1]{heinzner-schwarz-stoetzel}}]
  If $x \in X$ and $\mup(x) = 0$, there are a $G_x$-invariant
  decomposition $T_xX = \lieg \cd x \oplus W$, open $G_x$-invariant
  subsets $S \subset W$, $\Omega \subset X$ and a $G$-equivariant
  diffeomorphism $\Psi : G \times^{G_x}S \ra \Omega$, such that $0\in
  S, x\in \Omega$ and $\Psi ([e, 0]) =x$.
\end{teo}
Here $G \times^{G_x}S$ denotes the associated bundle with principal
bundle $G \ra G/G_x$. .
\begin{cor} \label{slice-cor} If $x \in X$ and $\mup(x) = \beta$,
  there are a $G^\beta$-invariant decomposition $T_xX = \lieg^\beta
  \cd x \, \oplus W$, open $G^\beta$-invariant subsets $S \subset W$,
  $\Omega \subset X$ and a $G^\beta$-equivariant diffeomorphism $\Psi
  : G^\beta \times^{G_x}S \ra \Omega$, such that $0\in S, x\in \Omega$
  and $\Psi ([e, 0]) =x$.
\end{cor}
This follows applying the previous theorem to the action of $G^\beta$
with the momentum map $\widehat{\mu_{\liu^\beta}} := \mu_{\liu^\beta} -
i\beta$, where $\mu_{\liu^\beta}$ denotes the projection of $\mu$ onto
$\mu_{\liu^\beta}$.
See \cite[p. 169]{heinzner-schwarz-stoetzel} for more details.

If $\beta \in \liep$, \changed{then $\beta_X$ is a vector field on
  $X$, i.e. a section of $TX$. For $x\in X$, the differential is a map
  $T_xX \ra T_{\beta_X(x)}(TX)$. If $\beta_X(x) =0$, there is a
  canonical splitting $T_{\beta_X(x)}(TX) = T_xX \oplus
  T_xX$. Accordingly $d\beta_X(x)$ splits} into a horizontal and a
vertical part. The horizontal part is the identity map. We denote the
vertical part by $d\beta_X(x)$.  It belongs to $\End(T_xX)$.  Let
$\{\phi_t=\exp(t\beta)\} $ be the flow of $\beta_X$.  \changed{There
  is a corresponding flow on $TX$. Since $\phi_t(x)=x$, the flow on
  $TX$ preserves $T_xX$ and there it is given by $d\phi_t(x) \in
  \Gl(T_xX)$.  Thus we get a linear $\R$-action on $T_xX$ with
  infinitesimal generator $d\beta_X(x) $.}
\begin{cor}
  \label{slice-cor-2}
  If $\beta \in \liep $ and $x \in X$ is a critical point of $\mupb$,
  then there are open invariant neighbourhoods $S \subset T_xX$ and
  $\Omega \subset X$ and an $\R$-equivariant diffeomorphism $\Psi : S
  \ra \Omega$, such that $0\in S, x\in \Omega$, $\Psi ( 0) =x$. \changed{(Here
  $t\in \R$ acts as $d\phi_t(x)$ on $S$ and as $\phi_t$ on $\Omega$.)}
\end{cor}
\begin{proof}
  The subgroup $H:=\exp(\R \beta)$ is compatible.  It is enough to
  apply the previous corollary to the $H$-action at $x$.
\end{proof}

Assume now that $\beta \in \liep$ and that $x \in
\Crit(\mu^\beta_\liep)$. Let $D^2\mup^\beta(x) $ denote the Hessian,
which is a symmetric operator on $T_xX$ such that
\begin{gather*}
  ( D^2 \mup^\beta(x) v, v) = \frac{\mathrm d^2}{\mathrm
    dt^2} 
  (\mup^\beta\circ \ga)(0)
\end{gather*}
where $\ga$ is a smooth curve, $\ga(0) = x$ and $ \dot{\ga}(0)=v$.
Denote by $V_-$ (respectively $V_+$) the sum of the eigenspaces of the
Hessian of $\mupb$ corresponding to negative (resp. positive)
eigenvalues. Denote by $V_0$ the kernel.  Since the Hessian is
symmetric we get an orthogonal decomposition
\begin{gather}
  \label{Dec-tangente}
  T_xX = V_- \oplus V_0 \oplus V_+.
\end{gather}
Let $\alfa : G \ra X$ be the orbit map: $\alfa(g) :=gx$.  The
differential $d\alfa_e$ is the map $\xi \mapsto \xi_X(x)$.
\begin{prop}
  \label{tangent}
  If $\beta \in \liep$ and $x \in \Crit(\mu^\beta_\liep)$ then
  \begin{gather*}
    D^2\mup^\beta(x) = d \beta_X(x).
  \end{gather*}
  Moreover $d\alfa_e (\lier^{\beta\pm} ) \subset V_\pm$ and $d\alfa_e(
  \lieg^\beta) \subset V_0$.  If $X$ is $G$-homogeneous these are
  equalities.
\end{prop}
\begin{proof}
  The first statement is proved in
  \cite[Prop. 2.5]{heinzner-schwarz-stoetzel}.  Denote by $\rho : G_x
  \ra T_xX$ the isotropy representation: $\rho(g) = dg_x$.  Observe
  that $\alfa$ is $G_x$-equivariant where $G_x$ acts on $G$ by
  conjugation, hence $d\alfa_e$ is $G_x$-equivariant, where $G_x$ acts
  on $\lieg$ by the adjoint representation and on $T_xX$ by the
  isotropy representation.  Since $\beta_X(x)=0$, $\exp({t\beta }) \in
  G_x$ for any $t$ and $d\alfa_e$ is $\R$-equivariant. Therefore it
  interchanges the infinitesimal generators of the $\R$-actions,
  i.e. $d\alfa_e \circ \ad \beta = d\beta_X = D^2\mupb(x)$.  The
  required inclusions follow. If $G$ acts transitively on $X$ we must
  have $T_xX = d\alfa_e(\lieg)$. Hence the three inclusions must be
  equalities.
\end{proof}

\begin{cor}
  \label{MorseBott}
  For every $\beta \in \liep$, $\mupb$ is a Morse-Bott function.
\end{cor}
\begin{proof}
  Let $X^\beta:=\{ x\in X: \beta_X(x) =0\}$.  Corollary
  \ref{slice-cor-2} implies that $X^\beta$ is \changed{a smooth
    submanifold}. Since $T_xX^\beta = V_0$ for $x\in X^\beta$, the
  first statement of Proposition \ref{tangent} shows that the Hessian
  is nondegenerate in the normal directions.
\end{proof}

\subsection{Coadjoint orbits}
\label{coadjoint-orbits}

Let $U$ be a compact connected semisimple Lie group.  Fix a scalar
product $\sx \ ,\ \xs$ on $\liu$ and identify $\liu^* \cong \liu$.
Let $z \in \liu$ and let $Z:=U\cd z$ (adjoint action).
$Z$ is a (co)adjoint, hence it is provided with the
Kostant-Kirillov-Souriau symplectic form which is defined by
\begin{gather*}
  \om_z ( v_Z, w_Z) := \sx x, [v,w]\xs \qquad v, w \in \liek.
\end{gather*}
(See e.g. \cite[p. 5] {kirillov-lectures}.)  The inclusion $ Z
\hookrightarrow \liu$ is the momentum map for the $U$-action on $Z$.
Set $Q:=(U^\C)^{z+}$. Then $Q$ is a parabolic subgroup of $U^\C$ and
$T_zZ \cong \liu^\C / \lieq$. This endows $Z$ with an invariant
complex structure $J$ such that $\om$ is an invariant K\"ahler
form. Such a structure is in fact unique.  The action of $U$ on $Z$
extends to a holomorphic action of $U^\C$.

To study $K$-orbits on $\liep$ it is convenient to identify $ \liep$
with $ i\liep$ by multiplying by $i$.  A $K$-orbit $\OO=K\cd x \subset
\liep$ is mapped to $K \cd ix \subset Z:= U\cd ix$. Since $G \subset
U^\C$, $G$ acts on $Z$ and we have $G \cd ix = K\cd ix$, see
\cite[Lemma 5]{heinzner-stoetzel-global} for the case $G^\C =U^\C$ and
\cite[Prop. 6]{heinzner-stoetzel} for the general case.  Therefore the
data $G, K, U, Z, X$ are like in the previous setting.  And
identifying $\OO \cong K\cd ix$, the gradient momentum becomes the
inclusion $\OO \subset \liep$.

\section{Face structure}
  \subsection{Faces as orbitopes}
  \label{faces1}
  Let $U $ be a compact Lie group and let $G \subset U^\C$ be a
  compatible connected subgroup.
  \begin{defin}
    An \enf{orbitope} of $G$ is the convex envelope of a $K$-orbit in
    $\liep$.  If $\OO \subset \liep$ is the $K$-orbit in $\liep$, $\c$
    denotes the corresponding orbitope.
  \end{defin}

  \begin{lemma} \label{extO=O} We have $\ext \c = \OO$ and $\ext F = F
    \cap \OO$ for any face $F$ of $ \c$.
  \end{lemma}
  \begin{proof}
    This fact is common to all orbitopes, see \cite
    [Prop. 2.2]{sanyal-sottile-sturmfels-orbitopes} or \cite[Lemma
    14]{biliotti-ghigi-heinzner-1-preprint}.
  \end{proof}

  We start the analysis of the structure of the faces of $\c$ by
  considering the exposed faces. At the end of \S \ref{faces-2} we
  will prove that in fact all faces of $\c$ are exposed.  Let $\beta$
  be a nonzero vector in $\liep$. Since $\mup$ is the inclusion $\OO
  \hookrightarrow \liep$, the function $\mup^\beta $ is $\mup^\beta
  (x) : = \sx x, \beta \xs$.  Set
  \begin{gather*}
    \ml(\beta) : =\{ x \in \OO: \mup^\beta( x) = \max_\OO \mup^\beta
    \}.
  \end{gather*}
  The main result about this set is the following.
  \begin{prop}\label{massimo-connesso}
    The set $\ml(\beta)$ is a connected $K^\beta$-orbit. In particular
    it is a $(K^\beta)^0$-orbit.
  \end{prop}
  This theorem goes back to \cite{duistermaat-kolk-varadarajan,
    heckman-thesis}. Since it is basic we repeat the proof in our
  context.  If $\liea \subset \liep$ is a maximal subalgebra, we
  denote by $W=W(\liek,\lia)$ the Weyl group of $\lia$ in $K$.

  \begin{lemma}
    \label{Heckman-roots}
    Let $\lieg$ be a real semisimple Lie algebra with Cartan
    decomposition $\lieg= \liek \oplus \liep$ and let $\lia \subset
    \liep$ be a maximal subalgebra. If $x, y \in \lia$ then there is a
    Weyl chamber $C$ such that $\cchamber$ contains both $x$ and $y$
    if and only if $\la (x) \la (y) \geq 0$ for every restricted root
    $\la$.
  \end{lemma}
  \begin{proof}
    [Proof (see \protect {\cite [p. 11] {heckman-thesis}})] A Weyl
    chamber is a connected component of the set where all roots are
    nonzero. Given such a component $C$, let $\roots_+$ be the set of
    roots that are positive on $C$. Then $\roots=\roots_+ \sqcup
    (-\roots_+)$. From this follows the ``only if'' part. To prove the
    ``if'' part \changed{we can assume that $x$ and $y$ are
      different. Let $z:= (x + y) /2 $ and let $C$ be a Weyl chamber
      with $z\in \cchamber$. By assumption, no root changes its sign
      on the segment $[x,y]$. Therefore $\la (z) >0$ implies that
      $\la(x) \geq 0$ and $\la(y) \geq 0$. If $\la(z) =0$, then
      $\la(x) = \la(y) = 0$. Therefore $x$ and $y$ belong to
      $\cchamber$. We thank the referee for pointing out this short
      argument.}
%
%
%
  \end{proof}

  \begin{lemma}
    \label{w-tilde}
    Let $\chamber \subset \lia$ be a Weyl chamber and let $x, y \in
    \cchamber$. If $x' \in \Weyl\cd x$, then there is a Weyl chamber
    $C'$ such that $ x' , y \in \overline{C'}$ if and only if there is
    $w \in \Weyl$ such that $ w \cdot x= x'$ and $w \cdot y=y$.
  \end{lemma}
  \begin{proof}
    The ``if'' part follows from the definition of a Weyl chamber.
    Assume the existence of a Weyl chamber $C'$ such that $ x' , y \in
    \overline{C'}$.  Then $x'=\sigma x$ for some $\sigma\in W$.  Let
    $w \in W$ be such that $w( C)=C'$. The points $w\meno x' = w\meno
    \sigma x \in $ and $x$ belong to $\overline{C}$ and to the same
    Weyl orbit. Hence $w\meno x' = w\meno \sigma x = x$
    \cite[p. 52]{humphreys-algebras}, i.e. $x'=w x$. Also $w\meno y $
    and $y $ belong to $\overline{C}$. Hence also $wy=y$.  This
    concludes the proof.
  \end{proof}

\begin{prop}\label{ss-works}
  Let $G$ be a real connected semisimple Lie group. Let $\beta \in
  \liep$.
  \begin{enumerate}
  \item If $\lia \subset \liep^\beta$ is a maximal subalgebra, then
    \begin{gather*}
      \liep^\beta = \bigcup_{k\in (K^\beta)^0} \Ad (k) \lia.
    \end{gather*}
  \item Let $W^\beta:= \{w\in W: w \beta = \beta \}$. Then for any
    $w\in W^\beta$ there is a $k\in (K^\beta)^0$ such that $\Ad(k)
    \lia = \lia$ and $\Ad (k) x = w \cd x $ for every $x\in \lia$.
  \end{enumerate}
\end{prop}
For a proof see for example \cite[p. 378-9, 383,
455-7]{knapp-beyond}).

\begin{lemma}
  $ \Crit(\mup^\beta) = \OO\cap \liep^\beta$.
\end{lemma}
\begin{proof}
\changed{Let $Z$ be the $U$-orbit containing $\OO$ as in \S \ref{coadjoint-orbits}.
  As observed in \S \ref {subsection-gradient-moment}
  $\operatorname{grad} \mupb = \beta_Z\restr{\OO}$.  So the set of critical points of
  $\mupb$ on $\OO$ is the set of zeros of $\beta_Z$ on $Z$ intersected with $\OO$. Since $(i\beta)_Z(x) =
  [i\beta, x]$, we have $\Crit(\mupb) = \OO \cap \liep^\beta$.}
\end{proof}

\begin{lemma}\label{critical-set}
  Let $G$ be semisimple. Fix $x\in \Crit(\mup^\beta)$. Let $\liea
  \subset \liep$ be a maximal subalgebra containing both $x$ and
  $\beta$. Then
  \begin{gather*}
    \Crit(\mup^\beta)=\bigcup_{w \in \Weyl} (K^\beta)^0 \cd w\cdot x =
    (K^\beta)^0 \cd N_K(\lia) \cd x,
  \end{gather*}
  where $\Weyl=\Weyl(\liek,\lia)$ is the Weyl group.
\end{lemma}
\begin{proof}
  Let $z\in \Crit(\mup^\beta) = \OO\cap\liep^\beta$. By Proposition
  \ref {ss-works} there is $k\in (K^\beta)^0$ such that $k\cdot z \in
  \liea$.  But $k\cdot z \in \OO$ and $\OO \cap \liea= \Weyl\cdot x$.
\end{proof}
\begin{prop}\label{hesso-2}
  Let $G$ be semisimple.  Assume that $x\in \OO\cap \lia$ and $ \beta
  \in \lia$.  Then $x$ is a local maximum of $\mup^\beta$ if and only
  if there exists a Weyl chamber $C \subset \lia$ such that $x,\beta
  \in \overline{C}$.
\end{prop}
\begin{proof}
  Let $\roots$ be the set of restricted roots of $(\lieg, \lia)$ and
  let $\xi = \xi_0 + \sum_{\la \in \roots}\xi_\la$ with $\xi_\la \in
  \lieg_\la$.  Fix a set of positive roots $\roots_+$ such that $\la
  (x ) \geq 0$ for every $\la \in \roots_+$.  We have
    \begin{gather*}
      \liek= \liez_\liek(\lia) \oplus \bigoplus _{\la \in \roots_+}
      \bigl ( \lieg_\la \oplus \lieg_{-\la}\bigr)\cap \liek.
    \end{gather*}
    (See e.g. \cite[p. 370]{knapp-beyond}.)  Since $T_x\OO = \liek \cd
    x = [\liek, x]$ and $[x, \lieg_\la] = \lieg_\la $ if $\la(x) \neq
    0$ and $[ x, \lieg_\la ] = 0$ otherwise, we have
    \begin{gather*}
      T_x\OO = \bigoplus _{\la(x)>0} \bigl ( \lieg_\la \oplus
      \lieg_{-\la}\bigr)\cap \liep.
    \end{gather*}
    If $w\in T_x\OO$, choose $\xi\in \liek$ such that $w=\xi_\OO(x) =
    [\xi, x]$ and set $\ga(t): = \Ad(\exp(t\xi)) \cd x$.  Then $\ga
    (0) = x$, $ \dot{\ga}(t)= [\xi, \ga(t)]$, $ \ddot{\ga}(0) = [\xi,
    [\xi, x]]$ and
  \begin{gather*}
    D^2\mup^\beta(x)(w,w) = \desudtzero \mup^\beta(\ga(t)) = \sx
    \ddot{\ga}(0), \beta \xs =-\sx [\xi,x], [\xi, \beta]\xs .
  \end{gather*}
  We can assume that $\xi = \sum_{\la(x)>0}\xi_\la$ with $\xi_\la \in
  \lieg_\la$.  This determines $\xi$ uniquely.  Then
  \begin{gather*}
    [x, \xi ] =\sum_{\la(x)>0} \la(x) z_\la
  \end{gather*}
  where $z_\la = \xi_\la - \xi_{-\la} $.  Since $\xi \in \liek$,
  $\theta(\xi_\la) = \xi_{-\la}$ and $z_\la \in \liep$.  Moreover the
  vectors $z_\la$ are orthogonal to each other.  Similarly $ [\beta,
  \xi ] =\sum_{\la \in \root_+} \la(\beta) z_\la $. So
  \begin{gather*}
    D^2\mup^\beta(x)(w,w) = -\sum_{\la(x)>0} \la (x) \la(\beta)
    |z_\la|^2.
  \end{gather*}
  If there is $\la \in \roots_+$ such that $\la(x)\la(\beta) < 0 $,
  then $x$ is not a local maximum point.  Otherwise the Hessian is
  negative semidefinite and $D^2\mupb(x)(w, w) =0$ if and only if $z_\la \neq 0
  \Rightarrow \la(\beta) =0$.  This means that the kernel of
  $D^2\mupb(x) $ is $\liek^\beta \cd x =T_x\Crit (\mupb)$. So the
  Hessian is degenerate only along the critical submanifold and is
  negative definite in the transverse direction. It follows that $x$
  is a local maximum point.  Summing up we have shown that $x$ is a
  local maximum point of $\mupb$ if and only if $\la (x) \la(\beta ) \geq 0$ for
  every $\la \in \roots$.  By Lemma \ref{Heckman-roots} this is
  equivalent to the condition that $x$ and $\beta$ lie in the closure
  of some Weyl chamber.  The result follows.
\end{proof}

\begin{proof}[Proof of Proposition \ref{massimo-connesso}]
  We start assuming that $G$ is semisimple.  Let $E$ be the set of all
  local maxima of $\mup^\beta$.  Since the function $\mup^\beta$ is
  $K^\beta$-invariant, the sets $E$ and $\ml(\beta)$ are
  $K^\beta$-invariant.  Since $\OO$ is compact there is at least a
  point $x\in \ml(\beta)$. Let $\lia \subset \liep $ be a maximal
  subalgebra containing $x$ and $\beta$. If $y\in E$, then by Lemma
  \ref{critical-set} there are $a \in (K^\beta)^0$ and $\tilde{w}\in
  \Weyl(\lieg, \lia)$ such that $y= a\cd \tilde{w} \cd x$. Since $y\in
  E$, also $\tilde{w}\cd x \in E$.  By Proposition \ref{hesso-2} there
  are Weyl chambers $C,C' \subset \lia$ such that $x, \beta \in
  \overline{C}$ and $w\cd x , \beta\in \overline{C'}$. By Lemma
  \ref{w-tilde} there is $w \in \Weyl$ such that $ w\cd x =
  \tilde{w}\cd x$ and $w \cd \beta = \beta$.  By Proposition
  \ref{ss-works} there is $k \in (K^\beta)^0$ such that $w \cd x =
  k\cd x$.
%
  It follows that $y\in (K^\beta)^0 \cdot x$. So $ E \subset
  (K^\beta)^0 \cdot x$. Since $(K^\beta)^0 \cd x \subset \ml(\beta)
  \subset E$ we conclude that $E=\ml(\beta) = (K^\beta)^0\cd x$.  In
  particular $\ml(\beta)$ is connected because it is an orbit of a
  connected group.  Since $\ml (\beta)$ is $K^\beta$-stable we also
  have $\ml(\beta) = K^\beta\cd x$.  If $G$ is not semisimple, then
  split $\lieg= \liez \oplus [\lieg, \lieg]$ with $\liez =
  \liez(\lieg)$. Accordingly $\liep=\liez\cap \liep \oplus
  \liep_{ss}$, $\liek = \liek \cap \liez \oplus \liek_{ss}$. Since $K$
  is connected, $K= \bigl (Z(G)\cap K\bigr )^0 \cd K_{ss}$.  If $\OO =
  K\cd x$ split $x= x_ 0 + x_1$ with $x_0\in \liez \cap \liep$ and
  $x_{1} \in \liep_{ss}$.  Then $\OO = x_0 + \OO_{1}$ where $\OO_{1}=
  K_{ss} \cd x_{1}$. If $\beta \in \liep$, split $\beta = \beta_0 +
  \beta _{1}$ with $\beta_0\in \liep\cap \liez$ and $\beta_1\in
  \liep_{ss}$. Then $\ml(\beta) = x_0 + \ml(\beta_{1})$.  By Lemma
  \ref{lemcomp} \ref{lemcomp2} $G_{ss}$ is a semisimple compatible
  subgroup of $U^\C$ and $\OO_1$ is a $K_{ss}$-orbit in
  $\liep_{ss}$. Therefore we know that $\ml(\beta_{1})$ is connected
  and that it is an orbit of both $(K_{ss}^{\beta_{1}})^0$ and
  $K_{ss}^{\beta_{1}}$. Since $K^\beta = \bigl (Z(G)\cap K \bigr )\cd
  K_{ss}^{\beta_{1}}$, we conclude that $\ml(\beta)$ is a connected
  orbit of $K^\beta$. Therefore it is also an orbit of $(K^\beta)^0$.
\end{proof}


\begin{cor}\label{hesso}
  Let $\beta$ be a nonzero vector in $ \liep$ and let $F_\beta(\c)$ be
  the exposed face of $\c$ defined by $\beta$, see
  \eqref{def-exposed}.  Then $\ext F_\beta(\c) = \ml(\beta) $,
  $F_\beta (\c) \subset \liep ^\beta$ and $\ext F_\beta(\c)$ is both a
  $K^\beta$ and a $(K^\beta)^0$-orbit.
\end{cor}
\begin{proof}
  By Lemma \ref{extO=O} $\ext F_\beta(\c) = \OO \cap F_\beta(\c) =
  \ml(\beta) $. Since $ \Crit(\mup^\beta) = \OO \cap \liep^\beta$, we
  see that $F_\beta(\c) \subset \liep^\beta$.  By Proposition
  \ref{massimo-connesso} $\ext F_\beta (\c) = \ml(\beta)$ is an orbit
  of $(K^\beta)^0$.
\end{proof}

\begin{prop}\label{facciona-orbita}
  Let $F$ be a nonempty face of $\c$. Then there is an abelian
  subalgebra $\lies \subset \liep$ such that $F$ is an orbitope of
  $(G^\lies)^0$, i.e.  $F \subset \liez_{\liep } (\lies)$ and $\ext F$
  is an orbit of $ (K^\lies)^0$.  If $F$ is proper, then $\lies \neq
  \{0\}$.
\end{prop}
\begin{proof}
  Fix a chain of faces $F=F_0 \subsetneq F_1 \subsetneq \cds
  \subsetneq F_k = \c$, such that for any $i$ there is no face
  strictly contained between $F_{i-1}$ and $ F_i$.  This is possible
  by Lemma \ref{face-chain}.  We will prove the result by induction on
  $k$. If $k=0$, then $F=\c$, so it is enough to set $\lies = \{0\}$.
  Let $k>1$ and assume that the theorem is proved for faces contained
  in a maximal chain of length $k-1$. Fix $F$ with a maximal chain as
  above of length $k$.  By the inductive hypothesis the theorem holds
  for $F_1$, so there is a nontrivial abelian subalgebra $\lies_1
  \subset \liep$ such that $F_1 \subset \liep^{\lies_1}$ and $\ext
  F_1$ is an orbit of $ (K^{\lies_1} )^0$.  In other words $F_1$ is an
  orbitope of $(G^{\lies_1})^0$, which is a compatible subgroup by Lemma
  \ref{lemcomp} \ref{lemcomp3}. Since $F$ is a maximal face of $F_1$,
  it is exposed. There is $\beta \in \liep^{\lies_1}$ such that
  $F=F_\beta (F_1)$. Set $ \lies = \lies_1 \oplus \R \beta$.  By
  Corollary \ref{hesso} $F \subset (\liep^\lies)^\beta = \liep^\lies$
  and $\ext F $ is an orbit of $((K^{\lies_1})^\beta)^0 =
  (K^\lies)^0$. Thus the inductive step is completed.  If $\lies =
  \{0\}$, then $(K^\lies)^0=K$, $\ext F = \OO$ and $F=\c$. So for
  proper faces $\lies\neq \{0\}$.
\end{proof}

\subsection{All faces are exposed}

\label{faces-2}

Let $G \subset U^\C$ be a compatible subgroup and let $\OO $ be a
$K$-orbit in $\liep$.  In general $\dim\c$ might be less than $\dim
\liep$ and there might be some normal subgroup of $K$ that acts
trivially on $\OO$. We wish to describe a decomposition of $G$ that is
useful in dealing with this degeneracy.  Let $A$ be \changed{the} affine hull of
$\OO$. This is an affine subspace of $\liep$ and we can write $A = x_0
+ \lipo$, where $\lipo \subset \liep$ is a linear subspace and $x_0
\in \liep$.  If we impose that $x_0\perp \lipo$, then $x_0$ is
uniquely determined.  It follows that $x_0$ is fixed by $K$. Hence by
Lemma \ref{convex-orbit} $x_0 \in \relint \c$.  Set also
\begin{gather*}
  \liek_1:=[\liep_1, \liep_1] \qquad
  \liep_0=\liep_1^\perp \qquad \liek_0=\liek_1^\perp \quad
  \lieg_1:=\liek_1 \oplus \liep_1 \qquad \lieg_0 : = \liek_0 \oplus
  \liep_0.
\end{gather*}
Thus $\liek = \liek_0\oplus \liek_1$ and $\liep=\liep_0 \oplus
\liep_1$ and $\lieg = \lieg_0 \oplus \lieg_1$.
\begin{prop}
  \label{deco-prop}
  $\lieg_1$ is a semisimple ideal of $\lieg$ and $\lieg_0$ is a
  reductive ideal. If $G_1$, $K_0$, $K_1$ are the corresponding
  analytic (connected) subgroups, then $G_1$ is compatible with $U^\C$
  and $K^0=K_0\cd K_1$.
  If $x\in \OO$, then $x=x_0 + x_1$ for some $x_1 \in \lipo$ and $\OO
  = x_0 + K_1\cd x_1$.
\end{prop}
\begin{proof}
  Since $\OO$ is a $K$-orbit, its affine hull is
  $K$-invariant. Therefore $x_0$ is fixed by $K$ and $[\liek, \lipo]
  \subset \lipo$.  It follows that $[\liek, \liko] =[\liek, [\lipo,
  \lipo]] = [\lipo, [\lipo, \liek]] \subset [\lipo, \lipo]=\liko$.
  Since $[\liek, \lipo]\subset\lipo$ and $[\liek, \liko]\subset\liko$
  also $[\liek, \lipt] \subset \lipt$ and $[\liek, \likt] \subset
  \likt$. Moreover $\sx [\lipo, \lipt], \liek\xs = B([\lipo, \lipt],
  \liek) = B( \lipt, [\liek, \lipo]) \subset B(\lipt, \lipo) = \sx
  \lipt, \lipo \xs =0$. ($B$ is the bilinear form defined \changed{at
the end of \S \ref{comp-subgrous}}.)
Since $[\lipo, \lipt]\subset \liek$ this means
  that $[\lipo, \lipt]=0$.  Using the Jacobi identity we get also
  $[\lipt, \liko] = [\lipt, [\lipo, \lipo]] = [\lipo, [\lipo,
  \lipt]]=0$.  Set $\lieg_1 :=\liek_1 \oplus \liep_1$. We have just
  showed that $\lieg_1$ is an ideal of $\lieg$.  Since it is
  $\theta$-invariant, $\lieg_1$ is a reductive subalgebra.  We claim
  that it is semisimple.  $\liek_1 \subset [\lieg_1, \lieg_1]$, so
  $\liez(\lieg_1) \subset \liep_1$.  Pick $x\in \OO$. We can split
  $x=x_0 + x_1 +x_2$ where $x_0 $ is as above, $x_2 \in
  \liez(\lieg_1)\cap \liep_1$, $x_1 \in \liep_1$ and $x_1 \perp
  \liez(\lieg_1)$. It follows that $\OO = x_0 +x_2 + K\cd x_1$, so the
  affine hull of $\OO$ is $x_0 +x_2 + \liep_1 \cap
  \liez(\lieg_1)^\perp$. Therefore $x_2 =0$ and $\liep_1 \cap
  \liez(\lieg_1)^\perp = \liep_1$, i.e. $\liez(\lieg_1)=\{0\}$. This
  proves that $\lieg_1$ is semisimple. Let $G_1\subset G$ the
  (connected) analytic subgroup tangent to $\lieg_1$. It is normal,
  closed \cite[p. 440]{knapp-beyond} and compatible by Lemma
  \ref{lemcomp} \ref{lemcomp3}.  The $B$-orthogonal complement of
  $\lieg_1$ is $\liek_0 \oplus \liep_0$, which is also an ideal. So
  $K=K_0\cd K_1$ where $K_1 = G_1 \cap U$ and $K_0$ is the analytic
  subgroup of $K$ tangent to $\liek_0$. Since $K_0$ and $K_1$ are
  normal commuting subgroups $K_0$ acts trivially on $\liep_1$. Hence
  $\OO = x_0 + K_1 \cd x_1$.
\end{proof}
This decomposition can be further refined by setting $\lieg_2 : =
[\lieg_0, \lieg_0]$ and $\lieg_3:=\liez(\lieg) = \liez(\lieg_0)$. They
are both $\theta$-invariant ideals of $\lieg$, $\lieg_2$ is \changed{semisimple} and
\begin{gather}
  \label{deco-3} \lieg=\lieg_1\stackrel{\perp}{\oplus} \lieg_2
  \stackrel{\perp}{\oplus} \lieg_3.
\end{gather}
Set $\liep_i:=\lieg_i\cap \liep$ and $\liek_i:=\lieg_i\cap \liek$.  At
the group level $K^0=K_1\cd K_2 \cd K_3$, where $K_i$ are the
corresponding analytic (connected) subgroups. Since $K\cd x_0=x_0$,
$x_0 \in \lieg_3$.

Let $\lia \subset \liep$ be a maximal subalgebra. Let $\pi: \liep \ra
\liea$ denote the orthogonal projection.  Set
\begin{gather*}
  P:=\pi(\OO).
\end{gather*}
The following convexity theorem of Kostant \cite{kostant-convexity} is
the basic ingredient in the whole theory.
\begin{teo}[Kostant] \label{Kostant} Let $x \in \liea\cap \OO$. Then $
  P=\mathrm{conv}(W\cdot x)$.  In particular, $ \polp$ is a convex
  polytope, $\ext \polp = \OO \cap \liea$ and $\ext \polp$ is a
  $W$-orbit.
\end{teo}
The original proof of Kostant assumes that $G$ is semisimple. One
easily reduces to that case using Proposition \ref{deco-prop}.  The
theorem can be proved within the framework of the gradient momentum
map \cite[Rmk. 5.4]{heinzner-schuetzdeller}.  Another approach is by
observing that the orbits of polar representations are isoparametric
submanifolds.  Terng \cite{terng-convexity} has proved a convexity
theorem for isoparametric submanifolds, which in the case of polar
orbits gives the original statement by Kostant. See also
\cite{palais-terng-LNM}.
The following lemma is a consequence of Kostant convexity theorem.
See \cite[Lemma 7]{gichev-polar} for a proof.
\begin{lemma}
  \label{gichev}
  (i) If $E \subset \liep$ is a $K$-invariant convex subset, then
  $E\cap \liea = \pi (E)$. (ii) If $A \subset \liea$ is a
  $\Weyl$-invariant convex subset, then $K \cdot A $ is convex and
  $\pi (K\cdot A )= A$.
\end{lemma}

\begin{prop} \label{proiezione-intersezione} Let $F $ be a face of
  $\c$. Choose a subalgebra $\lies \subset \liep$ such that $F$ be an
  orbitope of $(G^\lies)^0$. Let $\lia$ be a maximal subalgebra of
  $\liep$ containing $\lies$.  Set $ \sigma : = \pi(\ext F)$.  Then
  $\sigma = \pi(F) = F\cap \lia$ and $\sigma$ is a nonempty face of
  the polytope $\polp$.  If $F$ is proper, then $\sigma$ is proper.
  $F$ is an orbitope of $(G^{\sigma^\perp})^0$, where $\sigma^\perp
  \subset \lia$ denotes the orthogonal to the tangent space of
  $\sigma$.  Moreover $\ext F$ is an orbit of $K^{\sigma^\perp}$ and
  $F=K^{\sigma^\perp}\cd \sigma$.
\end{prop}
\begin{proof}
  The set $\ext F$ is an orbit of $(K^\lies)^0$ and $\lia \subset
  \lieg^\lies$.  By Kostant theorem $\pi(\ext F)= \conv(\ext F \cap
  \lia)$ and $\ext F\cap \lia$ is an orbit of the Weyl group
  $W=W(\lieg^\lies, \lia)$.  So $\sigma$ is convex.  Fix $x\in \ext F
  \cap \lia$.  Since $\pi$ is linear, $\pi(F)\subset \conv (\pi(\ext
  F)) =\sigma$. On the other hand $ \ext \sigma \subset \Weyl\cd x =
  (\ext F) \cap \liea$.  Hence $\sigma \subset F\cap \liea$.  And
  obviously $F\cap \liea \subset \pi(F)$. Summing up $\pi(F) \subset
  \sigma \subset F\cap \lia \subset \pi(F)$.  The first assertion is
  proved.  That $\sigma$ is a face of $\polp$ follows directly from
  Lemma \ref{micro-convesso}, while $\sigma=\pi(F) \neq \vacuo$ since
  $F\neq \vacuo$. To check the other assertions observe that $\ext F$
  is an orbit of $(K^\lies)^0$, so that we can apply Proposition
  \ref{deco-prop} to this orbit.  We get a semisimple normal subgroup
  $G_1$ of $ (G^\lies)^0$, a decomposition $\lieg^\lies = \lieg_1
  \oplus \lieg_2 \oplus \lieg_3$ like \eqref{deco-3} and compact
  subgroups $K_1$, $K_2$, $K_3 = Z(K^\lies)^0$ such that
  $(K^\lies)^0=K_1\cd K_2\cd K_3$.  It follows that $\lia = \lia_1
  \oplus \lia_2\oplus \liep_3$, where $\lia_i:=\lia \cap \lieg_i$ is a
  maximal subalgebra of $\liep_i$ for $i=1, 2$.  Moreover $\ext F =
  x_0 + K_1\cd x_1$, the affine hull of $F$ is $x_0 + \liep_1$ and
  $x_0 \in \relint F$.  The restriction of $\pi$ to $\liep_1$ is the
  orthogonal projection $\liep_1\ra \lia_1$ and the affine hull of
  $\sigma $ is $x_0 + \lia_1$.  Hence $\sigma^\perp = \lia_2\oplus
  \liep_3$.  $\lieg_1$ is semisimple and centralizes. Thus $\lies
  \subset \sigma^\perp$, $K^{ \sigma^\perp} \subset K^\lies$ and
  $(K^{\sigma^\perp})^0= K_1 \cd K_3$.  So $K_1 \subset
  K^{\sigma^\perp} \subset K^\lies$ and $ K_1\cd x \subset
  K^{\sigma^\perp} \cd x \subset K^\lies \cd x$.  Since $K_1 \cd x =
  K^\lies \cd x=\ext F$ we get that $\ext F$ is an orbit of
  $K^{\sigma^\perp}$.  But $\ext F$ is connected, so it is also an
  orbit of $(K^{\sigma^\perp})^0$.  Since $\sigma^\perp = \lia_2\oplus
  \liep_3$, $x _ 0 + \liep_1 \subset \liep_3 \oplus \liep_1 =
  \liep^{\sigma^\perp}$. This shows that $F$ is an orbitope of
  $(G^{\sigma^\perp})^0$.  We have to prove that $F = K^{\sigma^\perp}
  \cd \sigma$.  Since $K_2$ acts trivially on $x_0 + \liep_1$,
  $K^{\sigma^\perp} \cd \sigma = K^{\lies} \cd \sigma$. Since $F$ is
  $K^{\lies}$-invariant, we get $K^{\sigma^\perp} \cd \sigma \subset
  F$.  On the other hand $\ext F \subset K^{\lies} \cd \sigma$.  Since
  $\sigma$ is $W$-invariant we can apply Lemma \ref{gichev} (with
  $K=K^\lies$ and $\liep= \liep^\lies$) to get that $K^{\lies} \cd
  \sigma$ is convex. Therefore we get $F = K^{\lies} \cd
  \sigma=K^{\sigma^\perp} \cd \sigma$.  It remains to prove that
  $\sigma $ is proper, when $F$ is proper.  Assume first that the
  affine hull $\c$ is $\liep$. Then the affine hull of $\polp$ is
  $\lia$. If $F$ is proper, then $\lies \neq \{0\}$, so $\lia_1
  \subsetneq \lia$ and $\sigma\subsetneq P$.  In the general case, we
  have to apply Proposition \ref{deco-prop} this time to $\OO$ rather
  than $\ext F$.  $\c$ turns out to be a translate of an orbitope of a
  semisimple subgroup of $G$ by an element of the center of $\lieg$.
  $\lia$ splits into the center of $\lieg$ and a maximal subalgebra of
  the semisimple subgroup. With this we easily reduce to the case we
  have just considered.
\end{proof}

\begin{cor}
  \label{interseco-monotono}
  Let $F_1, F_2$ be  proper faces of $\c$, and let $\lies_1, \lies_2
  \subset \liep$ be subalgebras such that $F_i$ is a  $
  (G^{\lies_i})^0$-orbitope. Assume that $\lia \subset \liep$ is a
  maximal subalgebra containing both $\lies_1$ and $\lies_2$.  If $F_1
  \cap \liea = F_2 \cap \liea$, then $F_1 = F_2$.
\end{cor}
\begin{proof}
  If $\sigma :=F_i \cap \lia$, then $F_1 = K^{\sigma^\perp} \cd
  \sigma=F_2$.
\end{proof}

\begin{teo}
  \label{tutte-esposte}
  All proper faces of $\c$ are exposed.
\end{teo}
\begin{proof}
  Given a proper face $F\subset \c$ choose a subalgebra $\lies \subset
  \liep$ such that $F$ be a $(G^\lies)^0$-orbitope and choose a
  maximal subalgebra $\lia \subset \liep$ containing $\lies$.  By
  Proposition \ref{proiezione-intersezione} $\sigma : = F \cap \liea$ is a
  proper face of $\polp$.  Since all faces of a polytope are exposed
  \cite[p. 95]{schneider-convex-bodies}, there is a vector $\beta\in
  \liea$ such that $ \sigma = F_\beta (\polp)$.  Since $\beta \in
  \liea$ and $P=\pi(\OO)$, $h_\polp(\beta) =\max_{x\in \OO} \sx \beta,
  x\xs = h_\c (\beta)$.  Set $F' : = F_\beta(\c)$.  We wish to show
  that $F=F'$. The inclusion $F \subset F'$ is immediate: if $x \in
  F$, then $\pi(x) \in \sigma$, so $\sx x, \beta \xs = h_P(\beta)=
  h_\c(\beta)$. It is also immediate that $F'\cap \liea = \sigma $.
  So we have two faces $F$ and $F'$ with $F\cap \liea = F'\cap \liea =
  \sigma$.  Set $\lies':=\R \beta \subset \lia$. By Corollary
  \ref{hesso} $F'$ is an orbitope of $(G^{\lies'})^0$.  Applying
  Corollary \ref{interseco-monotono} we get $F=F'= F_\beta (\c)$.
\end{proof}
\begin{cor}
  If $\OO ' \subset \OO$ is a smooth submanifold, then $\conv (\OO')$
  is a face of $\c$ if and only if there is a vector $\beta$ such that
  $\OO'=\ml (\beta)$.
\end{cor}
\begin{proof}
  Set $F = \conv (\OO')$.  From the fact that $\OO$ is contained in a
  sphere, it follows as in Lemma \ref{extO=O} that $\ext F =
  \OO'$. Therefore the statement follows immediately from the fact
  that every face of $\c$ is exposed and from Lemma \ref{hesso}.
\end{proof}

\subsection{Faces and parabolic subgroups}
\label{parabolic-section}

In this section we prove Theorem \ref{main-2}, which follows
  from Propositions \ref{QF-para} and \ref{para2} below.  Given a
face $F \subset \c $ set
\begin{gather*}
  \HF :=\{ g\in K: gF=F\} = \{ g\in K: g\cd \ext F =\ext F \}\\
  Q_F : = \{ g\in G: g\cd \ext F =\ext F \} \qquad
  \CF : = \{ \beta \in \liep: F=F_\beta (\c)\} .
\end{gather*}
Denote by $ C_F^{H_F} $ the vectors of $C_F$ that are fixed by $H_F$.
\begin{prop}
  \label{gruppi-esposte}
  For any face $F$ the set $\ext F$ is an orbit of $\HF$.  If $F$ is
  proper, then $\CFI \neq \vacuo$. For any $\beta \in \CFI$, $\HF =
  K^\beta$ and $F\subset \liep^\beta$.
\end{prop}
\begin{proof}
  The group $\HF$ is compact. By Proposition \ref{facciona-orbita}
  $\ext F$ is an orbit of some subgroup $K'\subset K$. Hence
  $K'\subset H_F$ and $\ext F$ is an orbit also of $\HF$. It follows
  that $\HF$ preserves both $\c$ and $F$, so by Lemma \ref {u-cono}
  there is a vector $\beta\in \CF$ that is fixed by $\HF$. This proves
  that $\CFI \neq \vacuo$.  On the other hand given any $\beta \in
  C_F^{H_F}$, we have $H_F \subset K^\beta$ and $F=F_\beta(\c)$. By
  Lemma \ref{hesso}, $F \subset \liep^\beta$ and $\ext F =K^\beta \cd
  x$. It follows that $K^\beta \subset \HF$, hence $\HF =K^\beta$.
\end{proof}

\begin{lemma}
  \label{parabolic-included} Let $\lieq_1, \lieq_2$ be subalgebras of
  $\lieg$. Assume that $\lieq_1$ is parabolic, that $\lieq _1 \subset
  \lieq_2$ and that $\lieq_1\cap \liek = \lieq_2 \cap \liek$. Then
  $\lieq_1= \lieq_2$.
\end{lemma}
\begin{proof}
  Assume that $\lieq_1 = \lieg^{\beta+}$ for some $\beta \in \liep$.
  Then $\lieq_1 \cap \liek = \liek^\beta$.  Denote by $V_\la$ the
  eigenspace of $\ad \beta $ with eigenvalue $\la$. Then $\lieq_1 =
  \bigoplus_{\la\in J} V_\la$ where $J$ is the set of nonnegative
  eigenvalues of $\ad \beta$.  Since $\beta \in \lieq_1 \subset
  \lieq_2$, $\lieq_2$ is $\ad\beta$-stable.  We have
  \begin{gather*}
    \lieq_2 = \bigoplus_{\la\in I}\bigl ( V_\la\cap \lieq_2 \bigr)
  \end{gather*}
  for some set of eigenvalues $I$ and we can assume that $V_\la \cap
  \lieq_2 \neq \{0\}$ for every $\la \in I$. We wish to prove that $I
  \subset [0, \infty) $. If not there would be some negative $\la \in
  I$. Pick a nonzero $\xi \in V_\la \cap \lieq_2 $. Then $\theta(\xi)
  \in V_{-\la} \subset \lieq_1 \subset \lieq_2$. So $\xi + \theta(\xi)
  \in \lieq_2 \cap \liek$.  By assumption $\lieq_2\cap
  \liek=\lieq_1\cap \liek = \lieg^{\beta+}\cap\liek=\liek^\beta$. So
  we should have $[\beta, \xi + \theta(\xi)] = 0$, while $[\beta, \xi
  + \theta (\xi) ] = \la (\xi - \theta(\xi)) \neq 0$. The
  contradiction shows that $I \subset[0, \infty)$. So $I \subset J$
  and $\lieq_2 \subset \lieq_1$.
\end{proof}


\begin{prop}\label{QF-para}
  If $F \subset \c$ is a proper face, and $\beta \in C_F^{H_F}$, then
  $Q_F =G^{\beta+}$.
\end{prop}
\begin{proof}
  We prove first that $G^{\beta +} \subset Q_F$, i.e. that
  $G^{\beta+}$ preserves $\ext F$. Since $\beta \in C_F^{H_F}$,
  $H_F=K^\beta$.  In general $G^{\beta+}$ will not be connected.
  Nevertheless $K \cap G^{\beta+}=K^\beta$ meets all components of $
  G^{\beta+}$.  By Proposition \ref{gruppi-esposte} $ K^\beta =H_F
  \subset Q_F$. So it is enough to prove that $(G^{\beta+})^0 \subset
  Q_F$. This amounts to showing that for any $\xi \in \lieg^{\beta+}$
  the vector field $\xi_\OO$ is tangent to $\ext F$.  Fix an arbitrary
  $x\in \ext F$.  Since $F=F_\beta(\c)$, $\ext F = \ml(\beta)$, so $x$
  is a maximum point of $\mup^\beta$. Hence $V_+=\{0\}$ in
  \eqref{Dec-tangente}. By Proposition \ref{tangent}
  $d\alfa_e(\lieg^{\beta+} )= d\alfa_e(\lieg^\beta) +
  d\alfa_e(\lier^\beta_+) \subset V_0 + V_+ = V_0$. Hence for any $\xi
  \in \lieg^{\beta+}$, $\xi_\OO(x) = d\alfa_e(\xi) \in V_0= T_x \ext
  F$.  Thus we proved that $G^{\beta + } \subset Q_F$. We also know
  that $G^{\beta+} \cap K = K^\beta = H_F = Q_F\cap K$.  Also, $Q_F
  \subset G$ is a closed subgroup, hence a Lie subgroup.  Thus we can
  apply Lemma \ref{parabolic-included} to the Lie algebras of
  $G^{\beta+}$ and $Q_F$ respectively, and we obtain $\lieg^{\beta+} =
  \lieq_F$. Therefore $Q_F \subset N_G(\lieq_F) = G^{\beta+}$.  And
  thus the theorem is proved.
\end{proof}


\begin{prop}
  \label{para2}
  The set $\{\ext F: F $ a nonempty face of $\c \}$ coincides with the
  set of all closed orbits of parabolic subgroups of $G$.
Any parabolic subgroup $Q \subset G$ has a unique
  closed orbit, which equals the set of extreme points of a unique
  face of $F \subset \c$. If $Q=G^{\beta+}$, then $F=F_\beta(\c)$.
\end{prop}
\begin{proof}
  Let $Q \subset G$ be parabolic. There is at least one closed orbit
  since the action is algebraic.  Choose $\beta \in \liep$ such that
  $Q=G^{\beta+}$.  Then $K^\beta = Q\cap K$. Let $\OO'$ be any closed
  orbit of $Q$ and let $x \in \OO'$ be a maximum point of $\mup^\beta$
  over $ \OO'$. Since the gradient of $\mup^\beta $ at $x$ is $
  \beta_\OO(x)$ and $\beta \in \lieg^{\beta+}$, we get $\beta_\OO(x) =
  0$. By Proposition \ref{tangent} $d\alfa_e(\lieg^{\beta+}) =
  V_0\oplus V_+$, so $V_+ \subset T_x(G^{\beta+} \cd x ) =
  T_x\OO'$. Since $x$ is a maximum point of $\mup^\beta $ over $\OO'$,
  we conclude that $V_+=\{0\}$.  Thus $x$ is a local maximum point of
  $\mup^\beta$ and $R^{\beta+}$ acts trivially on $\OO'$.  But
  $\mup^\beta$ has only global maxima, hence $x\in \ml(\beta)$ and
  $\OO'=G^\beta\cd x = K^\beta \cd x=\ml(\beta)$.  Set
  $F=F_\beta(\c)$.  Then $\OO'=\ext F$. This proves that the closed
  orbit is unique.
\end{proof}

\begin{cor} \label{includiamo}
For any face $F$ we have
$    \CFI = \{\beta \in \liep: G^{\beta+} = Q_F\}$.
\end{cor}
\begin{proof}
  By Proposition \ref{QF-para} the set on the left is included in the set
  on the right.  Conversely, if $\beta$ is in the set on the right,
  then $\beta \in C_F$ with $F=F_\beta(\c)$, by the previous
  Theorem. Since $H_F=Q_F\cap K =G^{\beta+}\cap K = K^\beta$, $\beta$
  is also fixed by $H_F$.
\end{proof}

If $F$ is a proper face set
\begin{gather}
  \lies_F:=\spam(\CFI) \qquad G_F:=Q_F \cap \theta (Q_F).
\end{gather}
If $\beta \in \CFI$, then $G_F:=G^\beta$.
\begin{cor}\label{sf-zg}
  $ \lies_F$ is an abelian subalgebra of $\liep$ and $\lies_F =
  \liez(\lieg_F) \cap \liep$.
\end{cor}
\begin{proof}
  $\lies_F$ is the span of $\CFI$ and $\lieg_F = \lieq_F \cap \theta
  \lieq_F$. Thus the result follows from Corollary \ref{includiamo} and Lemma
  \ref{para-beta}.
\end{proof}
\begin{cor}\label{beta-regular}
  $H_F = K^{\lies_F}$ and $G_F = G^{\lies_F}$.
\end{cor}
\begin{proof}
  It follows from the discussion in the proof of Lemma
  \ref{para-beta}, that the vectors of $\CFI$ are regular in
  $\lies_F=\lia_I$, i.e. if a root vanishes on $\beta \in \CFI$, then
  it vanishes on the whole of $\lies_F$. Thus $K^{\lies_F} = K^\beta$
  and $G^{\lies_F} = G^\beta$.
\end{proof}
\begin{cor}\label{GF-orbitope}
  The face $F$ is an orbitope of $G_F^0$.
\end{cor}
\begin{proof}
  If $\beta \in \CFI$, then $F$ is a $(G^\beta)^0$-orbitope by
  Corollary \ref{hesso}.
\end{proof}
\begin{cor}
  Let $F$ be a face and let $\lia \subset \liep$ be a maximal
  subalgebra. Then $\CFI\cap \lia \neq \vacuo$ if and only if $\CFI \subset \lia$
  if and only if $\lia \subset \lieg_F$.
\end{cor}
\begin{proof}
  If $\beta \in \CFI\cap \lia$, then $[\beta, \lia]=0$. Since $\beta$
  is regular in $\lies_F$, we get $\lies_F\subset \lia$. Conversely,
  if $\lies_F\subset \lia$, then $\CFI\subset \lia$. Since $\lieg_F =
  \lieg^{\lies_F}$ the condition $\lies_F\subset \lia$ is equivalent
  to $\lia \subset \lieg_F$.
\end{proof}

\subsection{Proof of Theorem \ref{main}}
\label{polysection}

Fix a maximal subalgebra $\lia \subset \liep$.  Denote by $\faces$ the
set of proper faces of $\OO$ and by $\facesp$ the set of proper faces
of the polytope $P$.  If $F$ is a face of $\OO$ and $a\in K$, then
$a\cdot F$ is still a face, so $K$ acts on $\faces$.  Similarly
$\Weyl=W(\lieg,\lia)$ acts on $\facesp$.  We wish to show that $\faces
/ K \cong \facesp /W$.

\begin{lemma} \label{facce-1} For every face of $\c$ there is $a\in K$
  such that $\lies_{a\cd F} \subset \lia$. The face
  $a\cd F$ is unique up to $N_K(\liea)$.
\end{lemma}
\begin{proof}
  By Theorem \ref{tutte-esposte} $F=F_\ga (\c)$ and $H_F=K^\ga$ for
  some $\ga \in \liep$. Choose $a\in K$ such that $\Ad (a) \ga \in
  \lia$.  Then $a \cd F = F_{\Ad(a)\ga} (\c)$. Therefore $\Ad(a)\ga$
  belongs to $C_{a\cd F}^{H_{a\cd F}}$ and also to $\lia$. By
  Corollary \ref {beta-regular} $\lies_{a\cd F} \subset \lia$. To
  prove the second statement it is enough to show that if $F=F_\ga
  (\c)$ with $\ga \in \lia$ and $\Ad(a)\ga \in \lia$, then there is $g
  \in N_K(\lia)$ such that $g\cd F = a \cd F$.  Since $\ga \in
  \lia\cap \Ad(a\meno) \lia$, both $\lia$ and $\Ad(a\meno)\lia$ are
  maximal subalgebras in $\liep^\ga$. Hence there is $g\in K^\ga = H_F
  $ such that $\Ad(a\meno) \lia = \Ad(g) \lia$. Therefore $w: = a g
  \in N_K(\lia)$ and $a \cd F = ag \cd F = w \cd F$.
%
%
%
\end{proof}

Define a map
\begin{gather*}
  \phi : \faces / K \ra \facesp /\Weyl
\end{gather*}
by the following rule: given a class in $ \faces/K$ choose a
representative $F$ such that $\lies_F\subset \lia$ and set $\phi(
[F]): = [F\cap \liea]$.  By Proposition \ref{proiezione-intersezione}
$F\cap \liea$ is indeed a face of the polytope and by Lemma
\ref{facce-1} a different choice of the representative will yield the
same class in $\facesp/W$, so that the map $\phi$ is well-defined.

Now fix a face $F$ with $\lies_F \subset \lia$.  $F$ is an orbitope of
$G_F^0$. Applying Proposition \ref{deco-prop} we get a decomposition
$\lieg_F = \lieg_1 \oplus \lieg_2 \oplus \lieg_3$ like \eqref{deco-3}.
Here $\lieg_3 = \liez(\lieg_F)$.  Accordingly $\lia = \lia_1 \oplus
\lia_2\oplus \lies_F$, where $\lia_i:=\lia \cap \lieg_i$ is a maximal
subalgebra of $\liep_i$ for $i=1, 2$. We have used the fact that
$\liep_3 = \liez(\lieg_F) \cap \liep = \lies_F$ by Corollary
\eqref{sf-zg}.  Denote by $ W_1$ and $W_2$ the Weyl groups of
$(\lieg_1, \lia_1)$ and $(\lieg_2, \lia_2)$. They can be considered as
subgroups of $W=\Weyl(\lieg, \lia)$. They commute and have the
following sets of invariant vectors:
\begin{gather*}
  \liea^{W_1} = \liea_2 \oplus \lies_F \qquad \liea^{W_2} = \liea_1
  \oplus \lies_F \qquad \liea^{W_1 \times W_2} = \lies_F.
\end{gather*}

\begin{lemma} \label{facce-inv} Let $F \subset \c$ be a nonempty face
  with $\lies_F \subset \lia$. Set $ \sigma :=F \cap \liea$.  Then
  $W_1 \times W_2$ preserves $\sigma$.
\end{lemma}
\begin{proof}
  Recall from Proposition \ref{deco-prop} that $\ext F = x_0 + K_1 \cd
  x_1$. By Kostant theorem $\sigma =\pi(\ext F) = x_0 + \conv (W_1\cd
  x_1) = \conv (W_1\cd x)$. Hence $W_1$ preserves $\sigma$. Moreover
  $\sigma \subset \lies_F \oplus \liea_1 $ hence $W_2$ fixes $\sigma$
  pointwise and the statement follows.
%
\end{proof}

%
%

If $\sigma$ is a face of $P$ set $ G_\sigma : = \{g\ \in \Weyl: g
(\sigma ) = \sigma\}$.
\begin{lemma}
  If $\sigma \in \facesp$ there is a vector $\beta\in \liea$ that is
  fixed by $G_\sigma$ and such that $\sigma = F_\beta(P)$.  If $\beta$
  is any such vector and $F:=F_\beta (\c)$, then $F\cap \liea
  =\sigma$, $G_\sigma = W_1\times W_2$, $\lies_F = \liea^{G_\sigma}$
  and $F$ depends only on $\sigma$, not on the choice of $\beta$.
\end{lemma}
\begin{proof}
  The existence of a $G_\sigma$-invariant $\beta$ such that
  $F_\beta(P)= \sigma$ follows directly from Lemma \ref{u-cono}.  If
  $F :=F_\beta(\c)$ it follows immediately that $F\cap \lia = \sigma$.
  By Lemma \ref{facce-inv} $W_1 \times W_2 \subset G_\sigma$, so
  $\beta\in \lia^{G_\sigma}\subset \lia^{W_1 \times W_2} =
  \lies_F$. It follows that $H_F = K^\beta$.  The subgroup of $W$ that
  fixes $\beta$ is the Weyl group of $(\lieg^\beta, \lia)$
  i.e. $W_1\times W_2$. Hence $W_1\times W_2=G_\sigma$ and $\lies_F =
  \liea^{G_\sigma}$.  So $\lies_F$ depends only on $\sigma$, not on
  the choice of $\beta$.  The same holds for $H_F = K^{\lies_F}$ and
  for $\ext F $, which is equal to the $H_F$-orbit through a point in
  $\ext \sigma$.
\end{proof}

Define a map $ \psi : \facesp /\Weyl \ra \faces / K$ by the following
rule: given $\sigma$, fix $\beta\in \liea^{G_\sigma}$ such that
$\sigma = F_\beta(P)$ and set
$
  \psi( [\sigma]): = [F_\beta (\c)]$.
By the previous lemma $F_\beta(\c)$ depends only on $\sigma$, not on
$\beta$. It is clear that $\psi$ is well-defined on equivalence
classes.

\begin{yuppi}
  The maps $\psi$ and $\phi$ are inverse to each other. Therefore $
  \facesp /\Weyl$ and $ \faces / K$ are in bijective correspondence.
\end{yuppi}
\begin{proof}
  Let $\sigma$ be a face of $\polp$. Choose $\beta\in \lia^{G_\sigma}$
  such that $\sigma = F_\beta(P)$.  If $F:=F_\beta(\c)$, then $\lies_F
  \subset \liea$. So $\phi \circ \psi ([\sigma] ) = \phi ([F]) =
  [F\cap \liea ] = [\sigma]$ and $\phi \circ \psi$ is the identity.
  Thus $\phi$ is surjective. It is enough to show that $\phi$ is
  injective. Let $F_1, F_2 \subset \c$ be faces such that $\phi([F_1])
  = \phi([F_2])$.  Acting with $K$ we can assume that both
  $\lies_{F_1} $ and $\lies_{F_2}$ are contained in $\lia$. Acting
  with $W$ we can also assume that $F_1\cap \lia = F_2 \cap \lia$. By
  Corollary \ref{interseco-monotono} we get $F_1=F_2$. By Proposition
  \ref{proiezione-intersezione} the map between $ \facesp /\Weyl$ and
  $ \faces / K$ is the one stated in the introduction.
\end{proof}

\begin{remark}\label{polare}
  Let 
  $K_1 \ra \operatorname{O}(V)$ be a polar representation.  By Dadok's
  theorem there is a semisimple Lie group $G$ with Cartan
  decomposition $\lieg=\liek\oplus \liep$ such that $V=\liep$ and the
  orbits of $K_1$ coincide with the orbit of $\Ad K$. A maximal
  subalgebra $\lia \subset \liep$ is a section for both
  actions. Denote by $W$ the Weyl group of $(\lieg, \lia)$ and by
  $W_1$ the Weyl group of the polar representation of $K_1$.  If $x\in
  \lia$, then $W\cd x = K\cd x \cap \lia = K_1 \cd x \cap \lia = W_1
  \cd x$.  We claim that $\faces/K_1 = \faces/K$ and $\facesp/W_1 =
  \facesp/W$.  Indeed let $F \in \faces$ and $k\in K$. Fix a point
  $x\in \relint F$.  There is some $k_1 \in K_1$ such that $k_1 x =
  kx$. Then $kx$ belongs both to $\relint kF$ and to $\relint
  k_1F$. Hence $kF =k_1F$ by Theorem \ref{schneider-facce}. This shows
  that the $K$-orbit through $F$ is contained in the $K_1$-orbit
  through $F$. Interchanging $K$ and $K_1$ we get the opposite
  inclusion. Thus $\faces/K_1 = \faces/K$. In the same way one proves
  that $\facesp/W_1 = \facesp/W$. From this it follows that Theorem
  \ref{main} holds for any polar representation.
\end{remark}

\section{Final remarks}

It follows from the results in the previous section that there are a
finite number of $K$-orbits on the set $\faces$.  Given such an orbit,
we denote by $S$ the union of the faces in the orbit. Therefore $S$
equals $K\cd F$ for some face $F \in \faces$.  We call $S$ the
\emph{stratum} corresponding to the face $F$.  Arguing as in the case
of coadjoint orbitopes \cite[\S 5]{biliotti-ghigi-heinzner-1-preprint}
one proves the following.

\begin{teo}
  \label{stratification}
  The strata give a partition of $\partial \c$.  They are smooth
  embedded submanifolds of $\liep$ and are locally closed in $\partial
  \c$.  For any stratum $S$ the boundary $\overline{S} \setminus S$ is
  the disjoint union of strata of lower dimension.
\end{teo}
The computation of the dimension of the strata is trickier in this
case. Nevertheless the bound in the statement follows easily from the
following argument. If $E$ is an $n$-dimensional convex body, then
$\partial E$ has Hausdorff dimension $n-1$. If $F$ is an
$n$-dimensional face, the boundary of the stratum $S:=K\cd F$ is a
fiber bundle over a compact base with fibres isometric to $\partial
F$. Therefore its Hausdorff dimension is strictly smaller than the
dimension of $S$.

Also the description of the faces of $\c$ and of the momentum polytope
in terms of root data is just as in the case of coadjoint orbitopes
(see \S 6 in \cite{biliotti-ghigi-heinzner-1-preprint}). We briefly
state the result.

Fix a maximal subalgebra $\liea$ of $\liep$ and a system of simple
roots $\Pi \subset \roots= \roots (\lieg, \liea)$.  A subset $E\subset
\liea$ is \emph{connected} if there is no pair of disjoint subsets
$D,C\subset E$ such that $D\sqcup C =E$, and $\sx x,y \xs=0$ for any
$x \in D$ and for any $y \in C$.  (A thorough discussion of connected
subsets can be found in \cite{satake-compactifications}, \cite[\S
5]{moore-compactifications}.)  Connected components are defined as
usual.  If $x $ is a nonzero vector of $\liea$, a subset $I
\subset\simple$ is called $x$-\enf{connected} if $I\cup\{ x\}$ is
connected.  Equivalently $I \subset \simple$ is $x$-connected if and
only if every connected component of $I$ contains at least one root
$\alfa$ such that $\alfa (x) \neq 0$.  If $I\subset \simple$ is
$x$-connected, denote by $I'$ the collection of all simple roots
orthogonal to $\{ x\}\cup I$.  The set $J:=I\cup I'$ is called the
$x$-\enf{saturation} of $I$.  The largest $x$-connected subset
contained in $J$ is $I$. So $J$ is determined by $I$ and $I$ is
determined by $J$.  Given a subset $I\subset \simple$ we will denote
by $Q_I$ the parabolic subgroup with Lie algebra $\lieq_I$ as defined
in \eqref{para-dec}.
\begin{teo}
  \label{satakone}
  Let $\OO\subset \liep$ be a $K$-orbit and let $x $ be the unique
  point in $\OO \cap \cchamber$.
  \begin{enumerate}
  \item If $I \subset \simple$ is $x$-connected and $J$ is its
    $x$-saturation, then $Q_I\cd x = Q_J\cd x $ and $ F:= \conv ( Q_J
    \cd x) $ is a face of $\c$. If $\beta \in \lia_J$ and $\la(\beta)
    >0 $ for any $\la \in \simple \setminus J$, then $F=F_\beta
    (\c)$. Moreover $ Q_F = Q_J $.
  \item Any face of $\c$ is conjugate to one of the faces constructed
    in (a).
  \end{enumerate}
\end{teo}

\def\cprime{$'$}


\begin{thebibliography}{10}

\bibitem{biliotti-ghigi-2}
L.~Biliotti and A.~Ghigi.
\newblock {S}atake-{F}urstenberg compactifications, the moment map and
  $\lambda_1$.
\newblock {\em Amer. J. Math.} 135 (1):  237--274, 2013.

\bibitem{biliotti-ghigi-heinzner-1-preprint}
L.~Biliotti, A.~Ghigi, and P.~Heinzner.
\newblock {C}oadjoint orbitopes.
\newblock {\em to appear on Osaka J. Math.}

\bibitem{borel-ji-libro}
A.~Borel and L.~Ji.
\newblock {\em Compactifications of symmetric and locally symmetric spaces}.
\newblock Mathematics: Theory \& Applications. Birkh\"auser Boston Inc.,
  Boston, MA, 2006.

\bibitem{clerc-neeb}
J.-L. Clerc and K.-H. Neeb.
\newblock Orbits of triples in the {S}hilov boundary of a bounded symmetric
  domain.
\newblock {\em Transform. Groups}, 11(3):387--426, 2006.

\bibitem{dadok-polar}
J.~Dadok.
\newblock Polar coordinates induced by actions of compact {L}ie groups.
\newblock {\em Trans. Amer. Math. Soc.}, 288(1):125--137, 1985.

\bibitem{duistermaat-kolk-varadarajan}
J.~J. Duistermaat, J.~A.~C. Kolk, and V.~S. Varadarajan.
\newblock Functions, flows and oscillatory integrals on flag manifolds and
  conjugacy classes in real semisimple {L}ie groups.
\newblock {\em Compositio Math.}, 49(3):309--398, 1983.

\bibitem{gichev-polar}
V.~M. Gichev.
\newblock Polar representations of compact groups and convex hulls of their
  orbits.
\newblock {\em Differential Geom. Appl.}, 28(5):608--614, 2010.

\bibitem{heckman-thesis}
G.~Heckman.
\newblock {\em Projection of orbits and asymptotic behaviour of multiplicities
  of compact Lie groups}.
\newblock 1980.
\newblock PhD thesis.

\bibitem{heinzner-schuetzdeller}
P.~Heinzner and P.~Sch{\"u}tzdeller.
\newblock Convexity properties of gradient maps.
\newblock {\em Adv. Math.}, 225(3):1119--1133, 2010.

\bibitem{heinzner-schwarz-stoetzel}
P.~Heinzner, G.~W. Schwarz, and H.~St{\"o}tzel.
\newblock Stratifications with respect to actions of real reductive groups.
\newblock {\em Compos. Math.}, 144(1):163--185, 2008.

\bibitem{heinzner-stoetzel-global}
P.~Heinzner and H.~St{\"o}tzel.
\newblock Critical points of the square of the momentum map.
\newblock In {\em Global aspects of complex geometry}, pages 211--226.
  Springer, Berlin, 2006.

\bibitem{heinzner-stoetzel}
P.~Heinzner and H.~St{\"o}tzel.
\newblock Semistable points with respect to real forms.
\newblock {\em Math. Ann.}, 338(1):1--9, 2007.

\bibitem{helgason}
S.~Helgason.
\newblock {\em Differential geometry, {L}ie groups, and symmetric spaces},
  volume~80 of {\em Pure and Applied Mathematics}.
\newblock Academic Press Inc., New York, 1978.

\bibitem{humphreys-algebras}
J.~E. Humphreys.
\newblock {\em Introduction to {L}ie algebras and representation theory},
  volume~9 of {\em Graduate Texts in Mathematics}.
\newblock Springer-Verlag, New York, 1978.
\newblock Second printing, revised.

\bibitem{kirillov-lectures}
A.~A. Kirillov.
\newblock {\em Lectures on the orbit method}, volume~64 of {\em Graduate
  Studies in Mathematics}.
\newblock American Mathematical Society, Providence, RI, 2004.

\bibitem{knapp-beyond}
A.~W. Knapp.
\newblock {\em Lie groups beyond an introduction}, volume 140 of {\em Progress
  in Mathematics}.
\newblock Birkh\"auser Boston Inc., Boston, MA, second edition, 2002.

\bibitem{koranyi-remarks}
A.~Kor{\'a}nyi.
\newblock Remarks on the {S}atake compactifications.
\newblock {\em Pure Appl. Math. Q.}, 1(4, part 3):851--866, 2005.

\bibitem{kostant-convexity}
B.~Kostant.
\newblock On convexity, the {W}eyl group and the {I}wasawa decomposition.
\newblock {\em Ann. Sci. \'Ecole Norm. Sup. (4)}, 6:413--455 (1974), 1973.

\bibitem{miebach}
C.~Miebach. \newblock Geometry of invariant subsets in complex semi-simple Lie
groups. Dissertation, Ruhr-Universit\"at Bochum, 2007.



\bibitem{moore-compactifications}
C.~C. Moore.
\newblock Compactifications of symmetric spaces.
\newblock {\em Amer. J. Math.}, 86:201--218, 1964.

\bibitem{palais-terng-LNM}
R.~S. Palais and C.-L. Terng.
\newblock {\em Critical point theory and submanifold geometry}, volume 1353 of
  {\em Lecture Notes in Mathematics}.
\newblock Springer-Verlag, Berlin, 1988.

\bibitem{sanyal-sottile-sturmfels-orbitopes}
R.~Sanyal, F.~Sottile, and B.~Sturmfels.
\newblock Orbitopes.
\newblock {\em Mathematika}, 57:275--314, 2011.

\bibitem{satake-compactifications}
I.~Satake.
\newblock On representations and compactifications of symmetric {R}iemannian
  spaces.
\newblock {\em Ann. of Math. (2)}, 71:77--110, 1960.

\bibitem{schneider-convex-bodies}
R.~Schneider.
\newblock {\em Convex bodies: the {B}runn-{M}inkowski theory}, volume~44 of
  {\em Encyclopedia of Mathematics and its Applications}.
\newblock Cambridge University Press, Cambridge, 1993.

\bibitem{terng-convexity}
C.-L. Terng.
\newblock Convexity theorem for isoparametric submanifolds.
\newblock {\em Invent. Math.}, 85(3):487--492, 1986.

\end{thebibliography}
\end{document}